\documentclass[11pt]{amsart}
\usepackage{amsmath,amssymb,amscd,enumerate}
\usepackage[all]{xy}
\usepackage[pdftex]{graphicx,color}
\usepackage{url}

\usepackage[colorlinks,linkcolor=blue,citecolor=blue,pdfstartview=FitH]{hyperref}
\usepackage{tikz}
\usepackage[top=1.1in, bottom=1.1in, left=1.1in, right=1.1in]{geometry}

\usetikzlibrary{arrows,shapes,positioning}
\usetikzlibrary{decorations.markings}
\tikzstyle arrowstyle=[scale=1]
\tikzstyle directed=[postaction={decorate,decoration={markings,
    mark=at position .55 with {\arrow[arrowstyle]{stealth}}}}]
\tikzstyle ddirected=[postaction={decorate,decoration={markings,
    mark=at position .45 with {\arrow[arrowstyle]{stealth}},
    mark=at position .55 with {\arrow[arrowstyle]{stealth}}}}]
\tikzstyle reverse directed=[postaction={decorate,decoration={markings,
    mark=at position .55 with {\arrowreversed[arrowstyle]{stealth};}}}]
\tikzstyle reverse ddirected=[postaction={decorate,decoration={markings,
    mark=at position .55 with {\arrowreversed[arrowstyle]{stealth};},
    mark=at position .65 with {\arrowreversed[arrowstyle]{stealth};}}}]

\newcommand{\co}{\colon \thinspace}

\newcommand\be{\mathbf{e}}

\newcommand\bx{\mathbf{x}}
\newcommand\by{\mathbf{y}}
\newcommand\bu{\mathbf{u}}
\newcommand\bv{\mathbf{v}}
\newcommand\bw{\mathbf{w}}
\newcommand\bz{\mathbf{z}}

\theoremstyle{plain}
\newtheorem{para}{}[section]

\newtheorem{proposition}[para]{Proposition}
\newtheorem{lemma}[para]{Lemma}

\newtheorem{fact}[para]{Fact}

\theoremstyle{definition}

\newtheorem{definition}[para]{Definition}

\theoremstyle{remark}

\newtheorem*{remark}{Remark}
\newtheorem{remarks}{Remarks}

\begin{document}

\title[Trig of truncated triangles \& tetrahedra]{Trigonometry of partially truncated\\ hyperbolic triangles and tetrahedra}
\author{Jason DeBlois}
\address{Department of Mathematics, University of Pittsburgh}
\email{jdeblois@pitt.edu}

\begin{abstract} The first main results of this note establish forms of the hyperbolic laws of cosines and sines for certain classes of quadrilaterals and pentagons in the hyperbolic plane, having at least one ideal vertex and right angles at non-ideal vertices, in which the length of a horocyclic cross-section at an ideal vertex plays the role filled by the dihedral angle in the usual versions of these laws. The second set of main results concern transversal length, meaning the distance from a designated internal edge to its opposite, of partially truncated tetrahedra in three-dimensional hyperbolic space whose non-truncated vertices are ideal. Transversal lengths of such tetrahedra are proved to depend only on the entire collection of internal edge lengths (interpreted at ideal vertices in terms of horospherical cross-sections), and bounds on these lengths are established. The case of ideal tetrahedra (no truncated vertices) is also considered. All main results are established using the unifying perspective of the hyperboloid model and Lorentzian geometry. A thorough introduction to this perspective is provided, with references as appropriate.
\end{abstract}

\maketitle

This paper proves a family of trigonometric results about truncated and partially truncated hyperbolic triangles (in $\mathbb{H}^2$) and tetrahedra (in $\mathbb{H}^3$). Here in the two-dimensional setting  we produce an $n$-gon, for $n=4$, $5$ or $6$, by removing open neighborhoods of one, two, or all three vertices, respectively, of a triangle. Geometrically, we require each edge resulting from truncation to be at right angles to the others that it intersects. Similarly, a partially truncated \emph{tetrahedron} is obtained by truncating some or all vertices of a tetrahedron; and geometrically, we require each triangular face resulting from truncation to intersect other faces at right angles.

The results of Section \ref{PT} give lower bounds, or in some cases formulas, for transversal lengths of partially truncated tetrahedra whose non-truncated vertices are ideal.  A \textit{transversal length} of such a tetrahedron is the minimum distance between a specified pair of opposite \emph{internal} edges, meaning edges not produced by truncation. In order to capture the flavor of these results we restate the final one here, in slightly more accessible form. It pertains to (fully) ideal tetrahedra.

\theoremstyle{plain}
\newtheorem*{IdealTetTransProp}{Proposition \ref{ideal tet trans}}
\begin{IdealTetTransProp} For an ideal tetrahedron $\Delta\subset\mathbb{H}^3$ and a choice of horoballs $B_1,B_2,B_3,B_4$, one centered at each ideal vertex of $\Delta$, let $d_{ij}$ be the signed distance between $B_i$ and $B_j$ for each $i<j$ (with a negative sign if the horoballs overlap). If $\tilde\lambda_{12}$ is the geodesic joining the center of $B_1$ to that of $B_2$, and $\tilde\lambda_{34}$ is the geodesic joining the centers of $B_3$ to $B_4$, the length of the transversal of $\Delta$ joining $\tilde\lambda_{12}$ to $\tilde\lambda_{34}$ is given by
\[ \cosh T_4(x,y;a,b,c,d) = \frac{\sqrt{ad} + \sqrt{bc}}{\sqrt{xy}}, \]
where $x = e^{d_{12}}$, $y = e^{d_{34}}$, $a = e^{d_{13}}$, $b = e^{d_{14}}$, $c = e^{d_{23}}$, and $d = e^{d_{24}}$.
\end{IdealTetTransProp}

One sees from this that the transversal length depends only on the internal edge lengths, that it decreases with the length of either of the internal edges that it joins, and that it increases with the other four internal edge lengths. It is also invariant under swapping the index $B_1$ with $B_2$, and/or $B_3$ with $B_4$---these preserve the roles of $\tilde\lambda_{12}$ and $\tilde\lambda_{13}$ and act on the set $\{a,b,c,d\}$ of inputs by even involutions. Finally, it is invariant under changing the choice of horoball centered at any ideal vertex of $\Delta$, which rescales the distances along edges incident to that vertex, cf.~Remarks \ref{rescale horoball}. Thus although computing the transversal length requires choosing a set of horoballs, the length computed does not depend on the particular set chosen.

The prior results of Section \ref{PT}, which address transversal lengths of partially truncated tetrahedra with one, two, or three ideal vertices, as well as those of Section \ref{three} concerning \emph{fully} truncated tetrahedra in $\mathbb{H}^3$, have the same set of qualitative properties. In these results, the length of an edge joining two faces resulting from truncation, or the distance from a truncation plane to a horoball, replace the signed distance between horoballs as appropriate. In some cases, like that of fully truncated tetrahedra, we give lower bounds rather than explicit formulas for the transversal length, due to the complexity of the formulas involved. See eg.~Proposition \ref{tetra transversals}.

The results of Section \ref{three} here are applied in \cite{DeSha2} as one tool contributing to that paper's lower bounds on volumes of hyperbolic $3$-manifolds with totally geodesic boundary, and we expect those of Sections \ref{two} and \ref{PT} to be similarly useful in forthcoming work.

To prove these results we use the \textit{hyperboloid model} for $\mathbb{H}^n$, where it is taken as a subset of $\mathbb{R}^{n+1}$ equipped with the \textit{Lorentzian inner product}, a certain non positive-definite bilinear form. Vectors of the ambient $\mathbb{R}^{n+1}$ carry information about different objects of $\mathbb{H}^n$, depending on the sign of their self-pairing, and the Lorentzian inner product of two such vectors carries information about the hyperbolic distance between the objects the vectors encode. This property was exploited by Ratcliffe in \cite[Ch.~3]{Ratcliffe} to prove trigonometric formulas, and also in previous work, eg.~by Epstein--Penner \cite{EpstePen}. We describe it in Section \ref{background} and subsequently leverage it to encode partially truncated triangles using just three vectors, and truncated tetrahedra using just four, using the pairings of these vectors to prove our trigonometric formulas.


In Section \ref{two} we prove a set of hyperbolic trigonometric laws for partially truncated triangles in $\mathbb{H}^2$ that we do not know of in the standard references for such results, eg.~written by Fenchel \cite{Fenchel} (Chapter VI there has a vast collection), nor by Beardon \cite{Beardon} or Ratcliffe \cite{Ratcliffe}. In the results that we prove, the length of a horospherical cross-section at an ideal vertex plays the role that would be played by the dihedral angle at an ordinary vertex, and the distance to the cross-section plays the role of edge length. Proposition \ref{queue} pertains to hyperbolic quadrilaterals, and \ref{pee} to pentagons, in which the non-truncated vertices are \textit{ideal} (see Section \ref{background} below). There is even a set of laws for ideal triangles, reproduced below, for which we do not know a full reference.

\newtheorem*{IdealTriProp}{Proposition \ref{ideal tri}}
\newcommand\IdealTri{Let $\Delta$ be an ideal triangle in $\mathbb{H}^2$ and $B_1,B_2,B_3$ be horoballs, one centered at each ideal vertex of $\Delta$.  For $i\in\{1,2,3\}$, let $\theta_i$ be the length of the horocyclic arc $S_i\cap \Delta$, where $S_i = \partial B_i$, and for each $i<j$ let $d_{ij}$ be the signed distance between $B_i$ and $B_j$ along the geodesic joining their centers (with a negative sign if the horoballs overlap). Then:\begin{itemize}
	\item (Law of Sines) $\displaystyle{ \frac{\theta_1}{e^{d_{23}}} = \frac{\theta_2}{e^{d_{13}}} = \frac{\theta_3}{e^{d_{12}}} }$
	\item (First Law of Cosines) $\displaystyle{ \theta_1 = \sqrt{\frac{e^{d_{23}}}{e^{d_{12}}e^{d_{13}}} } }$
	\item (Second Law of Cosines) $\displaystyle{ e^{d_{23}} = \frac{1}{\theta_2\theta_3} }$ [This was proved as Lemma 3.3 of \cite{HoffPur}.] \end{itemize} }
\begin{IdealTriProp}\IdealTri\end{IdealTriProp}

The names above reference the corresponding results for \emph{compact} hyperbolic triangles, see eg. \cite[Theorems 3.5.2, 3.5.3, 3.5.4]{Ratcliffe}. We are not aware of other references in the literature for the main results of Sections \ref{two}, \ref{three}, and \ref{PT}, save for the ``Second Law of Cosines'' above. Section 2 of \cite{FrigePet} also considers the geometry of partially truncated tetrahedra, but tracks them using dihedral angles and is focused on different questions; eg.~existence and moduli.

\section{Background: the meaning of vectors in the hyperboloid model}\label{background}

We begin by reviewing Ratcliffe's notation from Chapter 3 of \cite{Ratcliffe}, which we will generally follow in describing the hyperboloid model of hyperbolic space. The \textit{Lorentzian inner product} of $\bx = (x_1,\hdots,x_{n+1})$ and $\by=(y_1,\hdots,y_{n+1})\in\mathbb{R}^{n+1}$ is defined as
\[ \bx\circ\by = -x_1y_1+x_2y_2+\hdots+x_{n+1}y_{n+1}, \]
and $\bx\ne\mathbf{0}$ is said to be \textit{space-like}, \textit{light-like}, or \textit{time-like} respectively as $\bx\circ\bx$ is positive, zero, or negative.  The \textit{Lorentzian norm} of $\bx$ is $\|\bx\| = \sqrt{\bx\circ\bx}$, where the square root is taken to be positive, zero, or positive imaginary in the respective cases above.  The \textit{light cone} is the set of light-like vectors, and its \textit{interior} is the set of time-like vectors.  A time-like or light-like vector is \textit{positive} if its first entry is. We note that the following version of the Cauchy-Schwartz inequality follows from the usual one, see eg.~formula (1.0.2) of \cite{DeB_Delaunay}: 

\begin{fact}\label{CS} For positive vectors $\bx$ and $\by$ with $\bx\circ\bx\leq 0$ and $\by\circ\by\leq 0$, $\bx\circ\by \leq -\sqrt{(\bx\circ\bx)(\by\circ\by)}$, with equality if and only if they are linearly dependent. 
\end{fact}

The \textit{hyperboloid model} $\mathbb{H}^{n}$ of hyperbolic space is the set of positive vectors with Lorentzian norm $i$ in $\mathbb{R}^{n+1}$, equipped with the distance $d_H$ defined by
\[ \cosh d_H(\bu,\bv) = -\bu\circ\bv. \]
It follows from Fact \ref{CS} that this formula is well-defined. It is the distance function determined by the Riemannian metric on $\mathbb{H}^{n}$ given, at each $\bx\in\mathbb{H}^n$, by restricting the Lorentzian inner product to $T_{\bx}\mathbb{H}^{n} = \bx^{\perp}\doteq \{\bv\,|\,\bv\circ\bx = 0\}$.  (This restriction is positive-definite since $\bx$ is time-like, see \cite[Theorem 3.1.5]{Ratcliffe}.) The isometry group of $\mathbb{H}^n$ is the group $O^+(1,n)$ of matrices preserving the Lorentzian inner product and the sign of time-like vectors, see \cite[\S 3.1]{Ratcliffe}, acting on $\mathbb{H}^n$ by restriction.

Given $\bx \in \mathbb{H}^n$ and a \textit{unit} space-like vector $\by$, ie.~with $\by \circ \by = 1$, if $\by \in T_{\bx} \mathbb{H}^n$ (recall that this means $\bx\circ\by = 0$) then defining $\gamma_{\by}(t) = \cosh t\ \bx + \sinh t\,\by$ determines a (unit-speed) geodesic in $\mathbb{H}^n$ with $\gamma_{\by}(0) = \bx$ and $\gamma_{\by}'(0)=\by$. For an arbitrary $\by\in T_{\bx}\mathbb{H}^n$,\begin{align}\label{arbitrary}
	\gamma_{\by}(t) \doteq \cosh \left(\|\by\| t\right)\,\bx + \frac{1}{\|\by\|}\sinh \left(\|\by\| t\right)\, \by \end{align}
is a constant-speed geodesic with $\gamma_{\by}(0) = \bx$ and $\gamma_{\by}'(0)=\by$. (This can be directly checked.) The exponential map of $\mathbb{H}^n$ based at $\bx$, a diffeomorphism $T_{\bx}\mathbb{H}^n \to \mathbb{H}^n$, is then given by $\by\mapsto\gamma_{\by}(1)$.

The most useful feature of the hyperboloid model for us is that vectors of $\mathbb{R}^{n+1}$ which are not time-like also encode geometric features of $\mathbb{H}^n$.  

\subsection{The meaning of light-like vectors}\label{light-like} Recall that $\bx\in\mathbb{R}^{n+1}$ is \textit{light-like} if $\bx\circ \bx = 0$. Any positive light-like vector $\bx$ is approached by a sequence of positive time-like vectors (for instance we can take $t\,\bx + (1-t)\,\be_1$ for $t$ approaching $1$ from below); hence its projective class $[\bx]$ in $\mathbb{R}P^n$ is approached by a sequence in the projectivization of $\mathbb{H}^n$. Conversely, the projectivization of the light cone is the frontier of the projectivization of $\mathbb{H}^n$ in $\mathbb{R}P^n$. In this sense we regard projectivized members of the light cone as \textit{ideal points} of $\mathbb{H}^n$.

Individual vectors in the positive light cone carry more specific information.

\begin{definition}\label{horosphere} The \textit{horosphere determined by} a positive light-like vector $\bx\in\mathbb{R}^{n+1}$ is $S = \{\bv\in\mathbb{H}^n\,|\,\bv\circ\bx = -1\}$.  The \textit{horoball bounded by $S$} is the set $B= \{\bv\in\mathbb{H}^n\,|\,\bv\circ\bx \geq -1\}$. We say that the projective class $[\bx]$ of $\bx$ is the \textit{ideal point} of $S$ or of $B$.
\end{definition}

A little multivariable calculus shows that the horosphere $S$ determined by a positive light-like vector $\bx\in\mathbb{R}^{n+1}$ is the smooth submanifold $f^{-1}(-1)$ of $\mathbb{H}^n$, where $f(\bu) = \bu\circ\bx$, and its tangent space at any $\bu_0\in S$ is $T_{\bu_0} S = \{\bv\in\mathbb{R}^{n+1}\,|\,\bv\circ\bu_0 = 0 = \bv\circ\bx\}$.  For any such $\bu_0$ one may check directly that the formula $F(\bv) = \bu_0 + \bv + \left(\frac{\bv\circ\bv}{2}\right)\bx$ defines a Riemannian isometry from $T_{\bu_0} S$, equipped with the restriction of the Lorentzian inner product, to $S\subset\mathbb{H}^n$.  Since the inner product's restriction is positive-definite on $T_{\bu_0} S$, this explicitly confirms the well known fact that $S$ is an isometrically embedded copy of the Euclidean space $\mathbb{R}^{n-1}$.  It also yields the following formula for the Euclidean distance $d_S(\bu_0,\bu_1)$ in $S$ between vectors $\bu_0$ and $\bu_1$:\begin{align}\label{euc dist}
	d_S(\bu_0,\bu_1) = \sqrt{-2(1+\bu_0\circ\bu_1)} \end{align}
To see this, set $F(\bv)$ equal to $\bu_1$ and solve for $\bv\circ\bv$ by taking the Lorentzian inner product of both sides with $\bu_0$.  Using the formula for $d_H(\bu_0,\bu_1)$ given above we obtain the comparison equation $d_S(\bu_0,\bu_1)/2 = \sinh (d_H(\bu_0,\bu_1)/2)$.  This implies in particular that the isometric embedding $F$ is proper; that is, $S$ has compact intersection with any compact set of $\mathbb{H}^n$.

\begin{lemma}\label{to the horoball!}  For $\bv\in\mathbb{H}^n$ and a positive light-like vector $\bx$, the signed hyperbolic distance $d$ from $\bv$ to the horosphere $S$ determined by $\bx$ satisfies $e^d = -\bv\circ\bx$, where the sign of $d$ is positive if $\bv$ lies outside the horoball $B$ bounded by $S$.  This distance is realized at $t = d$ on the geodesic
\[ \gamma(t) = e^{-t}\bv - \frac{\sinh t}{\bx\circ\bv} \bx = e^{-t}\bv + e^{-d}\sinh t\,\bx\quad \in\mathbb{H}^n, \]
which has $\gamma(0) = \bv$. We call $\gamma$ the \mbox{\rm geodesic through} $\bv$ \mbox{\rm in the direction of} $\bx$.\end{lemma}

\begin{remark} For any $t\in\mathbb{R}$, $\gamma'(t) = -\gamma(t) + e^{-d}e^t\bx$ is a linear combination of $\gamma(t)$ and $\bx$. Therefore by the discussion below Definition \ref{horosphere} it is normal to the horosphere through $\gamma(t)$ with ideal point $[\bx]$. 
\end{remark}

\begin{proof}  A vector $\bu\in\mathbb{R}^{n+1}$ lies in $S$ if and only if $\bu\circ\bu = -1$, so it lies in $\mathbb{H}^n$, and $\bu\circ\bx = -1$. By the theory of Lagrange multipliers, the restriction of $f(\bu)\doteq \bu\circ \bv$ to $S$ may attain a local extremum at $\bu\in S$ only if the gradient of $f$ at $\bu$ is a linear combination of the gradients of the constraint functions $g_1(\bu) \doteq \bu\circ\bx$ and $g_2(\bu)\doteq\bu\circ\bu$.  By a direct computation, $\nabla f(\bu) = \bar{\bv}$, $\nabla g_1(\bu) = \bar\bx$, and $\nabla g_2(\bu) = \bar\bu$, where $\bar\bv$ is obtained from $\bv$ by switching the sign of first entry, and similarly for the others.  It follows that at any local extremum of the restriction of $f$ to $S$, $\bv$ is a linear combination of $\bx$ and $\bu$.  

Since $\bv$, which is time-like, is not a multiple of $\bx$, which is light-like, this implies that we can express $\bu$ in terms of $\bv$ and $\bx$.  Upon plugging $\bu = a\bx + b\bv$ into the constraints and solving for $a, b\in\mathbb{R}$ we obtain the unique solution\begin{align}\label{at the horoball}
	\bu = \frac{1}{2}\left(1 - \frac{1}{(\bv\circ\bx)^2}\right)\bx -\frac{1}{\bv\circ\bx}\bv. \end{align}
The value of $f$ at $\bu$ is thus $\bu\circ\bv = \frac{1}{2}\left(\bv\circ\bx+\frac{1}{\bv\circ\bx}\right)$, so by the definition of the hyperbolic distance $d_H$ we have
\[ \cosh d_H(\bu,\bv) = \frac{1}{2}\left(-\bv\circ\bx + \frac{1}{-\bv\circ\bx}\right). \]
Therefore $e^{d_H(\bu,\bv)}$ is either $-\bv\circ\bx$ or its reciprocal, whichever is at least $1$ since $d_H(\bu,\bv)$ is non-negative.  If we take $d$ to be the \textit{signed} distance, with non-negative sign if $\bv$ is outside the horoball $B$, then by the definition of $B$ we have $e^d = -\bv\circ\bx$ in all cases.

We finally note that $d$ really is the (signed) distance from $\bv$ to $S$; that is, the unique critical point $\bu$ of $f$ described above is the global maximizer for the values of $f$ on $S$, so $d_H(\bx,\bu)$ is the global minimizer of distances from $\bv$ to points of $S$.  This follows from uniqueness and the fact that as $\bu\in\mathbb{H}^n$ escapes compact sets, $f(\bu)\to-\infty$.  Toward the latter point, note for an arbitrary $\bu = (u_1,\hdots,u_{n+1})\in\mathbb{H}^n$ that $u_1 = \sqrt{1+u_2^2+\hdots+u_{n+1}^2}$, so we can rewrite $f(\bu)$ as\begin{align*}
	f(\bu) & = -\sqrt{(1+u_2^2+\hdots+u_{n+1}^2)(1+v_2^2+\hdots+v_{n+1}^2)} + u_2v_2+\hdots+u_{n+1}v_{n+1} \\
	& = \frac{(u_2v_2+\hdots+u_{n+1}v_{n+1})^2 - (1+u_2^2+\hdots+u_{n+1}^2)(1+v_2^2+\hdots+v_{n+1}^2)}{\sqrt{(1+u_2^2+\hdots+u_{n+1}^2)(1+v_2^2+\hdots+v_{n+1}^2)} + u_2v_2+\hdots+u_{n+1}v_{n+1}}. \end{align*}
In passing from the first to the second line above we use the fact that $\sqrt{a}-\sqrt{b} = (a-b)/(\sqrt{a}+\sqrt{b})$.  Expanding the numerator, canceling certain terms, and rearranging yields:
\[ -1 - (u_2^2+\hdots+u_{n+1}^2) - (v_2^2 + \hdots + v_{n+1}^2) - \sum_{i\neq j} (u_i - v_j)^2. \]
The denominator is at most some fixed multiple of $\sqrt{1+u_2^2+\hdots+u_{n+1}^2}$, by the Cauchy-Schwarz inequality, whereas the numerator is at most the opposite of the square of this quantity.  So as claimed, $f(\bu)\to-\infty$ as $\bu$ escapes compact sets.

For the parametrized curve $\gamma$ defined in the statement, direct computation reveals that $\gamma(t)\circ\gamma(t) = -1$ for all $t$, so $\gamma$ maps into $\mathbb{H}^n$, and that $\gamma''(t) = \gamma(t)$.  Therefore $\gamma$ is a hyperbolic geodesic, by \cite[Theorem 3.2.4]{Ratcliffe}.  More direct computation shows that $\gamma(0) = \bv$ and $\gamma(d)$ is the nearest point $\bu$ to $\bv$ on $S$ described in (\ref{at the horoball}).
\end{proof}

\begin{lemma}\label{to the other horoball!}  For linearly independent positive light-like vectors $\bx_0$ and $\bx_1$ of $\mathbb{R}^{n+1}$, the minimum signed distance $d$ from points on $S_1$ to $S_0$ satisfies $e^d = -\frac{1}{2}\bx_0\circ\bx_1$, where $S_i$ is the horosphere of $\mathbb{H}^n$ determined by $\bx_i$ for $i=0,1$.  This distance is uniquely attained by points at $t = \pm d/2$ on the geodesic
\[ \gamma(t) = \frac{1}{\sqrt{-2(\bx_0\circ\bx_1)}}\left( e^t\,\bx_0 + e^{-t}\,\bx_1 \right) = \frac{1}{2}e^{-d/2} \left( e^t\,\bx_0 + e^{-t}\,\bx_1 \right) \]
from $\bx_1$ to $\bx_0$.\end{lemma}

\begin{proof}  A vector $\bu\in\mathbb{R}^{n+1}$ lies in $S_1$ if and only if $\bu\circ\bu = -1$, $\bu$ is positive, and $\bu\circ\bx_1 = -1$.  By the theory of Lagrange multipliers, the restriction of $f(\bu) = \bu\circ\bx_0$ to $B_1$ may attain a local extremum at $\bu\in S_1$ only if the gradient of $f$ at $\bu$ is a linear combination of the constraint gradients $\nabla g_1(\bu)$ and $\nabla g_2(\bu)$, where $g_1(\bu) = \bu\circ\bx_1$ and $g_2(\bu) = \bu\circ\bu$.  Direct computation yields $\nabla f(\bu) = \bar\bx_0$, $\nabla g_1(\bu) = \bar\bx_1$, and $\nabla g_2(\bu)=2\bar\bu$, where $\bar\bx_0$ is obtained from $\bx_0$ by multiplying the first entry by $-1$ and similarly for the others.  It thus follows that at such a local extremum $\bu$, $\bx_0$ is a linear combination of $\bx_1$ and $\bu$ so, since $\bx_0$ is not a multiple of $\bx_1$, $\bu$ is a linear combination of the $\bx_i$.

Plugging $\bu = a\bx_0 + b\bx_1$ into the constraint equations and solving for $a,b\in\mathbb{R}$ yields\begin{align}\label{at the other horoball}
	\bu = \frac{-1}{\bx_0\circ\bx_1}\bx_0 + \frac{1}{2}\bx_1 \end{align}
This is a positive vector since it is a positive linear combination of the positive vectors $\bx_0$ and $\bx_1$.  By Lemma \ref{to the horoball!} and a direct computation, the signed distance $d$ from $\bu$ to $S_0$ satisfies $e^d = -\frac{1}{2}\bx_0\circ\bx_1$.

Substituting $\bu$ for $\bv$ in the formula for the geodesic $\gamma(t)$ defined in Lemma \ref{to the horoball!} and simplifying yields
\[ \gamma(t) = \frac{e^t}{-\bx_0\circ\bx_1}\bx_0 + \frac{e^{-t}}{2}\bx_1. \]
Note that $\gamma(0) =\bu\in S_1$ and $\gamma(d)\in S_0$.  The more-symmetric formula given in the statement is obtained by translating the parametrization, replacing $t$ by $t+d/2$.

It remains to show for $\bu$ from the formula (\ref{at the other horoball}) that $f(\bu)$ is a global maximum of $f$ on $S_1$, hence that $d$ is a global minimum of the signed distance to $S_0$ on $S_1$.  This follows from the fact that $\bu$ is the unique critical point of $f$ on $S_1$, together with the fact that $f(\bv)\to-\infty$ as $\bv\in S_1$ escapes compact sets.  Indeed, for any fixed $r<0$, and any $\bv\in S_1$ such that $f(\bv) \geq r$, we have $\bv\circ\bu = -f(\bv)/\bx_0\circ\bx_1 - 1/2 \geq -r/\bx_0\circ\bx_1 - 1/2$, so $\bv$ is contained in the closed ball of radius $\cosh^{-1}(r/\bx_0\circ\bx_1+1/2)$ around $\bu$.  This ball is compact.\end{proof}

\subsection{The meaning of space-like vectors} Recall that $\by\in\mathbb{R}^{n+1}$ is \textit{space-like} if $\by\circ\by > 0$. We note that the orthogonal subspace $V = \{\bx\circ\by = 0\}$ to a space-like vector $\by$ is time-like, ie. containing a time-like vector, since if this were not so then $\mathbb{R}^{n+1}$ would have no time-like vectors. This motivates:

\begin{definition}\label{polar} The \textit{polar hyperplane} to a space-like vector $\by$ is $P = \{\bx\in\mathbb{H}^n\,|\, \bx\circ\by = 0\}$.\end{definition}

As defined in \cite[\S 3.2]{Ratcliffe}, a \textit{hyperplane} of $\mathbb{H}^n$ is its intersection with a time-like, codimension-one vector subspace of $\mathbb{R}^{n+1}$. Corollary 4 of \cite[\S 3.2]{Ratcliffe} implies that the group of hyperbolic isometries acts transitively on the set of hyperplanes. Thus each hyperplane is the polar hyperplane to a space-like vector, since for instance $(\mathbb{R}^n\times \{0\})\cap\mathbb{H}^n$ is the polar hyperplane to $\be_{n+1} = (0,\hdots,0,1)$. 

Every hyperplane $P = V\cap\mathbb{H}^n$ is a totally geodesic copy of $\mathbb{H}^{n-1}$ in $\mathbb{H}^n$, being, for any $\bx\in P$, the image of the restriction of the exponential map based at $\bx$ to $T_{\bx} P = V\cap\bx^{\perp}$. Conversely, the exponential map's explicit description shows that any $(n-1)$-dimensional totally geodesic subspace $P$ of $\mathbb{H}^{n}$ is contained in $V = \mathrm{span}\{\bx,T_{\bx}P\}$ for any $\bx\in P$, and hence is a hyperplane.

We define a \textit{half-space} to be the closure of one component of $\mathbb{H}^n - P$, for a hyperplane $P$. We call $P$ the \textit{boundary} of $H$ and $H-P$ the \textit{interior}. From eg.~the model case above we see that each hyperplane bounds exactly two distinct half-spaces, which have disjoint interiors.

\begin{lemma}\label{space-like corresp} There is a bijective correspondence between half-spaces of $\mathbb{H}^n$ and unit space-like vectors of $\mathbb{R}^{n+1}$ that sends $\by\in\mathbb{R}^{n+1}$ to $H = \{\bx\in\mathbb{H}^n\,|\,\bx\circ\by\le 0\}$. In the other direction, it sends a half-space $H$ to the unit outward normal $\by$ to $H$ at any point of its boundary.
\end{lemma}

Above, given a hyperplane $P$ and any $\bx\in P$, a \textit{normal} vector to $P$ at $\bx$---and to a half-space $H$ bounded by $P$---is an element of $T_{\bx}\mathbb{H}^n$ orthogonal to the codimension-one subspace $T_{\bx} P$. A unit normal vector $\by$ to $P$ determines a geodesic $\gamma_{\by}(t) = \cosh t\, \bx + \sinh t\, \by$ that intersects $P$ transversely, and we say $\by$ is \textit{outward} to $H$ if $\gamma(t) \in H$ for all $t<0$. 

\begin{proof} For a hyperplane $P$ and any $\bx\in P$, since the orthogonal subspace to $T_{\bx} P$ in $T_{\bx}\mathbb{H}^n$ is one-dimensional there are exactly two unit normals to $P$. If $\by$ is one of these, the other is $-\by$, and exactly one of them is outward to a given half-space $H$ bounded by $P$. Take this to be $\by$. Any $\bx'\in P$ is of the form $\gamma_{\bz}(1)$ for some $\bz\in T_{\bx} P$, with $\gamma_{\bz}$ as in (\ref{arbitrary})---ie.~$\bx'$ is the exponential image of $\bz$---and hence $\by\circ \bx'$ also equals $0$. Thus $P$ is the polar hyperplane of $\by$.

For this $\by$, we claim that $H = \{\bx\in\mathbb{H}^n \,|\,\bx\circ\by < 0\}$. Defining $f\co\mathbb{H}^n\to\mathbb{R}$ by $f(\bx) = \bx\circ\by$, note that since the interior of $H$ is a connected component of the complement of $P = f^{-1}(0)$, it maps into one of $(-\infty,0)$ or $(0,\infty)$ under $f$. Since it contains $\gamma_{\by}(t)$ for $t<0$, it is the former. Similarly, the other component of $\mathbb{H}^n-P$ maps into $(0,\infty)$, so the claim holds.

Conversely, a unit space-like vector $\by$ belongs to $T_{\bx}\mathbb{H}^n = \bx^{\perp}$ at any point $\bx$ of its polar hyperplane $P$, and it is normal to $T_{\bx} P = V\cap\bx^{\perp}$ for $V = \{\bv\in\mathbb{R}^{n+1}\,|\,\bv\circ\bx = 0\}$. A computation shows that it is also the outward normal to the half-space $H = \{\bx\in\mathbb{H}^n\,|\,\bx\circ\by\le 0\}$. \end{proof}

We use this to give a series of geometric interpretations on the Lorentz pairing between vectors of various types and space-like vectors. The first follows directly from Theorem 3.2.12 of \cite{Ratcliffe}.

\begin{lemma} For $\bv\in\mathbb{H}^n$ and a unit space-like vector $\by$, the signed distance $d$ from $\bv$ to the polar hyperplane of $\by$ satisfies $\sinh d = \bv \circ \by$, where the sign is negative if and only if $\bv$ is contained in the interior of the half-space bounded by $P$ with outward normal $\by$. \end{lemma}

In the next result and below, the \textit{ideal boundary} of a hyperplane $P = V\cap\mathbb{H}^n$ (respectively, a half-space $H$ bounded by $P$) is the intersection of $V$ (resp.~the closure of the component of $\mathbb{R}^{n+1}-V$ containing the interior of $H$) with the positive light cone.

\begin{lemma}\label{to the hyperplane!}  For a positive light-like vector $\bx\in\mathbb{R}^{n+1}$, let $S$ be the horosphere determined by $\bx$. Suppose $P\subset\mathbb{H}^n$ is a hyperplane with ideal boundary not containing $\bx$, and let $\by\in\mathbb{R}^{n+1}$ be the outward-pointing normal to the half-space $H$ bounded by $P$ with ideal boundary containing $\bx$.  Then $\bx\circ\by < 0$, and the minimal signed distance $h$ from $P$ to $S$ satisfies $e^h = -\bx\circ\by$, uniquely realized by $\gamma(0)\in P$ and $\gamma(h)\in S$ for
\[ \gamma(t) =  e^{-h}\cosh t\, \bx + e^{-t}\,\by. \]
This is the unique geodesic perpendicular to $P$ in the direction of $\bx$, in the sense of Lemma \ref{to the horoball!}.\end{lemma}

\begin{remark} In the complementary case to Lemma \ref{to the hyperplane!} in which $\bx$ as above lies in the ideal boundary of $P$, then for any $\bv\in P$ the entire geodesic $\gamma$ from $\bv$ in the direction of $\bx$ from Lemma \ref{to the horoball!} lies in $P$. Thus $P$ contains points at arbitrarily small signed distance from $B$; in particular, it intersects it.
\end{remark}

\begin{proof}  A vector $\bv\in\mathbb{R}^{n+1}$ lies in $P$ if and only if $\bv\circ\bv = -1$, $\bv$ is positive,  and $\bv\circ\by = 0$.  By the theory of Lagrange multipliers, the restriction of $f(\bv)\doteq \bv\circ \bx$ to $P$ may attain a local extremum at $\bv\in P$ only if the gradient of $f$ at $\bv$ is a linear combination of the constraint gradients $\nabla g_1(\bv)$ and $\nabla g_2(\bv)$, where $g_1(\bv) = \bv\circ\by$ and $g_2(\bv) = \bv\circ\bv$.  Direct computation yields $\nabla f(\bv) = \bar\bx$, $\nabla g_1(\bv) = \bar\by$, and $\nabla g_2(\bv) = 2\bar\bv$, where $\bar\bx$ is obtained from $\bx$ by multiplying the first entry by $-1$ and similarly for the others.  It thus follows that $\bx$ is a linear combination of $\by$ and $\bv$ for such a point $\bv$, so since $\bx$ is not a multiple of $\by$ we can express $\bv$ in terms of $\bx$ and $\by$.

Plugging $\bv = a\bx+b\by$ into the constraint equations and solving for $a,b\in\mathbb{R}$ yields:\begin{align}\label{at the hyperplane}
	\bv = \pm\,\left( \frac{-1}{\bx\circ\by}\bx + \by \right) \end{align}
Only one of these two solutions is positive.  We claim that $\bv$ is positive and hence is the unique critical point of the restriction of $f$ to $H$.  By Lemma \ref{to the horoball!} its signed distance $h$ to $B$ will then satisfy $e^h = -\bx\circ\by$, and the geodesic through $\bv$ in the direction of $\bx$ will be given by:
\[ \gamma(t) = e^{-t}\bv -\frac{\sinh t}{\bx\circ\bv}\bx = \frac{\cosh t}{-\bx\circ\by}\bx + e^{-t}\by=  e^{-h}\cosh t\, \bx + e^{-t}\,\by. \]

To prove the claim, we first note that $\bx\circ\by < 0$: this follows from the fact that the half-space $H$ whose ideal boundary contains $\bx$ is characterized as $H = \{\bv\in\mathbb{H}^n\,|\,\bv\circ\by \le 0\}$. We then write $\bx = (x_1,\bx_0)$ and $\by = (y_1,\by_0)$ for vectors $\bx_0,\by_0\in\mathbb{R}^n$, so the first entry of $\bv$ is $x_1/(-\bx\circ\by) + y_1$. The hypothesis that $\bx$ is positive means that $x_1 > 0$, so since $\bx\circ\by < 0$, the first entry of $\bv$ is certainly positive if $y_1 \ge 0$. We therefore suppose that $y_1 < 0$. Since $\bx$ is light-like and $\by$ is unit space-like, we can write $x_1 = \|\bx_0\|$ and $y_1 = -\sqrt{\|\by_0\|^2-1}$, and hence
\[ \bx\circ\by = \|\bx_0\|\sqrt{\|\by_0\|^2-1} + \bx_0\cdot\by_0, \]
where $\bx_0\cdot\by_0$ is the ordinary dot product of $\bx_0$ and $\by_0$. Since $\bx\circ\by < 0$ we must have $\bx_0\cdot\by_0 < 0$; by the Cauchy Schwarz inequality, $-\bx_0\cdot\by_0 \leq \|\bx_0\|\|\by_0\|$. 
Thus we have:\begin{align*}
	\frac{-1}{\bx\circ\by}x_1 + y_1 & = \frac{\|\bx_0\|}{-\bx_0\cdot\by_0-\|\bx_0\|\sqrt{\|\by_0\|^2 -1}} - \sqrt{\|\by_0\|^2-1} \\
	& \ge \frac{\|\bx_0\|}{ \|\bx_0\|\|\by_0\|-\|\bx_0\|\sqrt{\|\by_0\|^2 -1}} - \sqrt{\|\by_0\|^2-1} \end{align*}
Simplifying the above and using the fact that $1/(\|\by_0\| - \sqrt{\|\by_0\|^2-1}) = \|\by_0\| + \sqrt{\|\by_0\|^2-1}$, we obtain in this case that $x_1/(-\bx\circ\by) + y_1\ge \|\by_0\| > 0$. This proves the claim.



It remains to show that $\bv$ is the global maximizer for the restriction of $f$ to $P$, hence that it is the minimizer for the signed distance to $S$.  This follows from the fact that $\bv$ is the unique critical point of the restriction of $f$ to $P$, together with the fact that $f(\bu)\to-\infty$ as $\bu\in P$ escapes compact sets.  Indeed, for any fixed $r<0$ and $\bu\in P$ such that $\bu\circ \bx > r$, we have $\bu\circ\bv = (-1/\bx\circ\by)\bu\circ\bx > -r/\bx\circ\by$; hence $\bu$ lies in the closed ball of radius $\cosh^{-1}(r/\bx\circ\by)$ about $\bv$. 
\end{proof}

The result below combines a few recorded by Ratcliffe in \cite{Ratcliffe}.

\begin{lemma}[cf.~\cite{Ratcliffe}, pp.~65--69]\label{geometry of inner product}  Let $\by_1, \by_2 \in \mathbb{R}^{n+1}$ be linearly independent space-like vectors, with polar hyperplanes $P_1$ and $P_2$ in $\mathbb{H}^n$, contained in $n$-dimensional subspaces $V_1$ and $V_2$ of $\mathbb{R}^{n+1}$, respectively.  Exactly one of the following holds:  \begin{enumerate}
\item $P_1$ and $P_2$ intersect in $\mathbb{H}^n$, and $|\by_1\circ\by_2| < \|\by_1\|\|\by_2\|$. Hence for some $\eta(\by_1,\by_2) \in (0,\pi)$:
$$\by_1\circ \by_2 = \|\by_1\| \|\by_2\| \cos \eta(\by_1,\by_2).$$  
For any $\bv \in P_1 \cap P_2$, $\eta(\by_1,\by_2)$ is the angle in $T_{\bv} \mathbb{H}^n$ between the normal vectors $\by_1$ and $\by_2$ to $P_1$ and $P_2$, respectively, at $\bv$. 
\item The distance between points of $P_1$ and $P_2$ attains a non-zero minimum, and $|\by_1\circ\by_2| > \|\by_1\|\|\by_2\|$. Hence for some $\eta(\by_1,\by_2) \in (0,\infty)$:
$$ |\by_1\circ \by_2| = \|\by_1\| \|\by_2\| \cosh \eta(\by_1,\by_2).  $$
In this case $\eta(\by_1,\by_2)$ is the (minimum) distance in $\mathbb{H}^n$ between $P_1$ and $P_2$, and $\by_1\circ \by_2 < 0$ if and only if $\by_1$ and $\by_2$ are oppositely oriented tangent vectors to the hyperbolic geodesic intersecting each of  $P_1$ and $P_2$ perpendicularly.
\item $P_1\cap P_2 = \emptyset$ but their ideal boundaries intersect, and $|\by_1\circ\by_2| = \|\by_1\|\|\by_2\|$.\end{enumerate} 
\end{lemma}

In case (3) above we say that $P_1$ and $P_2$ are \textit{parallel}. One can show in this case that there are sequences in $P_1$ and $P_2$ such that the infimum of distances from points of the first sequence to points of the second is $0$. We now expand on case (2) above.

\begin{lemma}\label{perp point} Suppose $\by_1$ and $\by_2$ are linearly independent space-like vectors such that the distance between points of their polar hyperplanes $P_1$ and $P_2$ attains a non-zero minimum. This distance  is realized as $d(\bv_1,\bv_2)$ for unique $\bv_1 \in P_1$ and $\bv_2 \in P_2$, with $\bv_1$ given by: \begin{align*}
  \bv_1 = \pm \frac{\left(\by_1\circ \by_2/\|\by_1\|\right)\by_1 - \|\by_1\| \by_2}{\sqrt{(\by_1\circ \by_2)^2 - \|\by_1\|^2\|\by_2\|^2}}, \end{align*}
where the sign of ``$\pm$'' above is negative if $\bv_1$ belongs to the half-space $H_2$ bounded by $P_2$ with $\by_2$ as outward normal vector, and positive otherwise.\end{lemma} 

\begin{proof} Standard facts of hyperbolic geometry imply the uniqueness of $\bv_1\in P_1$ and $\bv_2\in P_2$, and furthermore that the geodesic $\gamma$ joining $\bv_1$ and $\bv_2$ intersects each of $P_1$ and $P_2$ perpendicularly. Therefore $\gamma$ has tangent vector $\by_1$ at $\bv_1$ and $\by_2$ at $\bv_2$, and it follows that $\gamma = \mathrm{Span}\{\by_1,\by_2\} \cap \mathbb{H}^n$.  Taking $\bv_1 = a\by_1 + b\by_2$ and solving the equations $\bv_1 \circ \by_1 = 0$ and $\bv_1 \circ \bv_1 = -1$ (necessary for $\bv_1 \in \mathbb{H}^n$) for $a$ and $b$ yields the two solutions above. Taking an inner product with $\by_2$ now yields
\[ \bv_1\circ\by_2 = \pm \frac{1}{\|\by_1\|} \frac{\left(\by_1\circ \by_2\right)^2 - \|\by_1\|^2 \|\by_2\|^2}{\sqrt{(\by_1\circ \by_2)^2 - \|\by_1\|^2\|\by_2\|^2}} \]
By Lemma \ref{space-like corresp}, $\bv_1$ belongs to the half-space $H_2$ with $\by_2$ as outward normal if and only if $\bv_1\circ\by_2 < 0$, hence if and only if the ``$\pm$'' above is negative.\end{proof}

\section{Dimension two}\label{two}

Here we prove trigonometric formulas for a hyperbolic quadrilateral with two ideal vertices and a hyperbolic pentagon with one ideal vertex, each with right angles at all finite vertices.

\begin{proposition}\label{queue}  Let $Q\subset\mathbb{H}^2$ be a convex quadrilateral with a single compact side of length $\ell$ and right angles at its endpoints, and let $B_0$ and $B_1$ be horoballs centered at the two ideal vertices of $Q$.  If $a_i$ is the signed distance to $B_i$ from the other endpoint of the half-open edge of $Q$ containing the ideal point of $B_i$, $i=0,1$, and $d$ is the signed distance from $B_0$ to $B_1$, then
\[ \sinh(\ell/2) = e^{(d-a_0-a_1)/2}. \]
If $\theta_i$ is the length of the horocyclic arc $S_i\cap Q$, $i=0,1$, where $S_i = \partial B_i$, then for each $i$,
\begin{align*}  \frac{\theta_0}{e^{a_1}} = \frac{\theta_1}{e^{a_0}} = \frac{\sinh\ell}{2e^d} \end{align*}
\end{proposition}

\begin{proof} For a quadrilateral $Q\subset\mathbb{H}^2$ with a single compact edge $\gamma$ and right angles at the endpoints of this edge, let $\bx_0$ and $\bx_1$ be positive light-like vectors determining the horobolls $B_0$ and $B_1$ centered at the ideal vertices of $Q$.  Using the fact that the geodesic containing $\gamma$ is a codimension-one hyperplane of $\mathbb{H}^2$, let $\by$ be the space-like vector Lorentz-orthogonal to this geodesic with the property that $\bx_i\circ\by<0$ for $i=0,1$.  (Since the ideal vertices of $Q$ are on the same side of this geodesic, the inner products with $\by$ have the same sign by Lemma \ref{to the hyperplane!}.)

Let $\bv_0$ and $\bv_1$ be the finite vertices of $Q$, numbered so that $\bv_i$ is an endpoint of the half-open edge of $Q$ with its other endpoint at the center of $B_i$, for $i=0,1$.  Since $Q$ is right-angled, $\bv_i$ is described in terms of $\bx_i$ and $\by$ by the formula (\ref{at the hyperplane}) for each $i$.  (Note that there is a unique geodesic ray perpendicular to the geodesic containing $\gamma$ with its ideal endpoint at the center of $B_i$, since there is no hyperbolic triangle with two right angles.)  That is:\begin{align*}
	& \bv_0 = \frac{-1}{\bx_0\circ\by}\bx_0 + \by && \bv_1 = \frac{-1}{\bx_1\circ\by}\bx_1 + \by \end{align*}
By Lemma \ref{to the hyperplane!} their signed distances $a_i$ to the $B_i$ satisfy $e^{a_i} = -\bx_i\circ\by$ for $i=0,1$.  If $\ell$ is the length of $\gamma$ then from the distance formula we obtain
\[ \cosh\ell = -\bv_1\circ\bv_2 = \frac{-\bx_0\circ\bx_1}{(\bx_0\circ\by)(\bx_1\circ\by)} + 1 \]
It follows from Lemma \ref{to the other horoball!} that the minimal signed distance $d$ from $S_0$ to $S_1$ satisfies $e^d = -\frac{1}{2}\bx_0\circ\bx_1$, hence by a half-angle formula $\sinh(\ell/2) = e^{(d-a_0-a_1)/2}$ as claimed.

Let $\bu_0$ and $\bu_0'$ be the points of intersection between the horosphere $S_0 = \partial B_0$ and the edges of $Q$ joining the class of $\bx_0$ to $\bv_0$ and the class of $\bx_1$, respectively.  We obtain an explicit description for $\bu_0$ by plugging in $t = a_0$ to the parametrized geodesic $\gamma(t)$ starting at $\bv_0$ given in Lemma \ref{to the hyperplane!}, and for $\bu_0'$ by plugging in $t=d/2$ to the parametrized geodesic $\lambda(t)$ from $\bx_1$ given in Lemma \ref{to the other horoball!}.  These yield:\begin{align*}
	& \bu_0 = \frac{1}{2}\left(1+\frac{1}{(\bx_0\circ\by)^2}\right)\bx_0 + \frac{-1}{\bx_0\circ\by} \by &
	& \bu_0' = \frac{1}{2}\bx_0 + \frac{-1}{\bx_0\circ\bx_1}\bx_1 \end{align*}
From the horospherical distance formula we thus have
\[ \theta_0 = d_{S_0}(\bu_0,\bu_0') = \sqrt{-2(1+\bu_0\circ\bu_0')} = \sqrt{\frac{1}{(\bx_0\circ\by)^2} - \frac{2(\bx_1\circ\by)}{(\bx_0\circ\bx_1)(\bx_0\circ\by)}} \]
A similar computation yields an analogous formula for $\theta_1$, and we observe that\begin{align*}
	\theta_0 e^{-a_1} & = \theta_1 e^{-a_0} = \sinh\ell/(2e^d) \\
		& = \frac{1}{(\bx_0\circ\by)(\bx_1\circ\by)}\sqrt{\frac{2(\bx_0\circ\by)(\bx_1\circ\by)-\bx_0\circ\bx_1}{-\bx_0\circ\bx_1}} \end{align*}
The latter assertion in the statement follows.
\end{proof}

\begin{proposition}\label{pee}  Let $P\subset\mathbb{H}^2$ be a convex pentagon with four right angles and one ideal vertex, and let $B$ be a horoball centered at the ideal vertex of $P$.  Let $d$ be the length of the side of $P$ opposite its ideal vertex, let $\bw_0$ and $\bw_1$ be its endpoints, and for $i=0,1$ let $\ell_i$ be the length of the other side containing $\bw_i$.  If $\bv_i$ is the other endpoint of this side and $a_i$ is its signed distance to $B$, for $i=0,1$, then
\[ \cosh \ell_i = \frac{e^{a_i}\cosh d + e^{a_{1-i}}}{e^{a_i}\sinh d}\quad \mbox{for}\ i=0,1. \]
Moreover, if $\theta$ is the length of the horocyclic arc $S\cap P$, where $S=\partial B$, then
\[ \frac{\theta}{\sinh d} = \frac{\sinh\ell_0}{e^{a_1}} = \frac{\sinh\ell_1}{e^{a_0}}. \] \end{proposition}

\begin{proof}  Let $P$ be a pentagon with four right angles and a single ideal vertex, and let $\bx$ be a positive light-like vector that determines a horosphere $S$ centered at the ideal vertex of $P$.  Labeling the endpoints of the edge $\delta$ of $P$ opposite its ideal vertex as $\bw_0$ and $\bw_1$, for $i=0,1$ let $\gamma_i$ be the other edge of $P$ containing $\bw_i$, and let $\by_i$ be a unit space-like vector in $\mathbb{R}^3$ orthogonal to the geodesic containing $\gamma_i$.  Choose the $\by_i$ so that $\by_i\circ\bx<0$ for each $i$.  Equivalently, by Lemma \ref{to the hyperplane!}, $\by_i$ is the outward normal to the half-space $H_i$ bounded by $\gamma_i$ and containing $\bx$ in its ideal boundary.  This half-space also contains $\gamma_{1-i}$ in its interior; therefore by construction, $\by_0$ and $\by_1$ are oppositely-pointing tangent vectors to $\delta$, so by Lemma \ref{geometry of inner product}(2), $\by_0\circ\by_1 < 0$.

Lemma \ref{perp point} gives the formula below for the $\bw_i$, bearing in mind that for each $i$, $\bw_i\in H_{1-i}$.
\[ \bw_i = \frac{-(\by_0\circ\by_1)\by_i + \by_{1-i}}{\sqrt{(\by_0\circ\by_1)^2-1}}, \]
Let us call $\bv_i$ the endpoint of $\gamma_i$ not equal to $\bw_i$, for $i=0,1$.  An explicit formula for $\bv_i$ is given by (\ref{at the hyperplane}), with $\by$ there replaced by $\by_i$.  For, say, $i=0$ we thus have
\[ \bw_0\circ\bv_0 = \frac{\by_0\circ\by_1 - (\bx\circ\by_1)/(\bx\circ\by_0)}{\pm\sqrt{(\by_0\circ\by_1)^2-1}} = \frac{-(\bx\circ\by_0)(\by_0\circ\by_1)+\bx\circ\by_1}{-(\bx\circ\by_0)\sqrt{(\by_0\circ\by_1)^2-1}} \]

If $\ell_i$ is the length of $\gamma_i$ and $a_i$ is the distance from $\bv_i$ to $S$, for $i=0,1$, and $d = d_H(\bw_0,\bw_1)$ is the length of the side opposite the ideal vertex, then the above equation becomes
\[ \cosh \ell_0 = \frac{e^{a_0}\cosh d +e^{a_1}}{e^{a_0}\sinh\ell} \]
This is because $\cosh\ell = -\bw_0\circ\bv_0$ by definition, $d_H(\bv_i,S) = -\bx\circ\by_i$ by Lemma \ref{to the horoball!}, and as can be explicitly checked, $\cosh d = -\bw_0\circ\bw_1 = -\by_0\circ\by_1$. The derivation of the formula for $\cosh\ell_1$ is analogous, and we have proved the hyperbolic law of cosines.

For the law of sines we first note that the point of intersection $\bu_i$ between $S$ and the geodesic from $\bv_i$ in the direction of $\bx$ is given by the formula (\ref{at the horoball}), with $\bv$ there replaced by $\bv_i$, for $i=0,1$.  From direct calculation and/or Lemma \ref{to the hyperplane!} we have $\bv_i\circ\bx = \by_i\circ\bx$, whence for each $i$ we have\begin{align*}
	 \bu_i = \frac{1}{2}\left(1+\frac{1}{(\bx\circ\by_i)^2}\right)\bx + \frac{-1}{\bx\circ\by_i} \by_0 \end{align*}
From this we obtain the following formula for the length $\theta$ of the horocyclic arc $S\cap P$:
\[ \theta = \sqrt{-2(1+\bu_0\circ\bu_1)} = \frac{\sqrt{(\bx\circ\by_0)^2+(\bx\circ\by_1)^2-2(\by_0\circ\by_1)(\bx\circ\by_0)(\bx\circ\by_1)}}{(\bx\circ\by_0)(\bx\circ\by_1)} \]
Direct computation now establishes this case of the hyperbolic law of sines.
\end{proof}

\begin{proposition}\label{ideal tri}\IdealTri\end{proposition}

\begin{proof} Let $\bx_1$, $\bx_2$, and $\bx_3$ be positive light-like vectors respectively determining $B_1$, $B_2$, and $B_3$, and for each $i< j$ let $\lambda_{ij}(t) = \frac{1}{2}e^{-d_{ij}/2}\left(e^t \bx_i + e^{-t}\bx_j\right)$ be the geodesic from $\bx_j$ to $\bx_i$ as in Lemma \ref{to the other horoball!}. By that result, the points of intersection $\bu_{12} = \lambda_{12}\cap S_1$ and $\bu_{13} = \lambda_{13}\cap S_1$ are given as:
\[ \bu_{12} = \frac{1}{2}\left(\bx_1 + e^{-d_{12}}\bx_2\right)\quad\mbox{and}\quad \bu_{13} = \frac{1}{2}\left(\bx_1 + e^{-d_{13}}\bx_3\right). \]
Using the fact that for $i\ne j$, $\bx_i\circ\bx_j = -2e^{d_{ij}}$ (again by Lemma \ref{to the other horoball!}), we obtain 
\[ \bu_{12}\circ\bu_{13} = -1 - \frac{e^{d_{23}}}{2e^{d_{12}}e^{d_{13}}} \quad\Rightarrow\quad \theta_1 = \sqrt{\frac{e^{d_{23}}}{e^{d_{12}}e^{d_{13}}} }. \]
The formula for $\theta_1$ above comes from (\ref{euc dist}). It is the ``First Law of Cosines" above. Formulas for $\theta_2$ and $\theta_3$ are completely analogous, and from these we obtain the ``Law of Sines'':
\[ \frac{\theta_1}{e^{d_{23}}} = \sqrt{\frac{1}{e^{d_{12}}e^{d_{13}}e^{d_{23}}} } = \frac{\theta_2}{e^{d_{13}}} = \frac{\theta_3}{e^{d_{12}}} \]
For the ``Second Law of Cosines'' above we simply multiply the formulas for $\theta_2$ and $\theta_3$ and solve for $e^{d_{23}}$.
\end{proof}

\section{Dimension three: transversals of truncated tetrahedra}\label{three}

We turn now to dimension three, in which hyperplanes are \textit{planes}, ie.~two-dimensional totally geodesic copies of $\mathbb{H}^2$. Here we consider the \textit{truncated tetrahedron} determined by a collection of pairwise disjoint and non-parallel planes $P_1$, $P_2$, $P_3$, $P_4$ such that for each $i$, a single half-space $H_i$ bounded by $P_i$ contains $P_j$ for all $j\ne i$. We define truncated tetrahedra and their \textit{transversals} in \ref{trunch teth}, and subsequently prove trigonometric formulas about transversal lengths.

It is a standard fact, proved in eg.~\cite[Lemma 2.3]{DeSha}, that for any collection of three disjoint planes in $\mathbb{H}^3$ there is a unique fourth plane meeting each of the original three at right angles. The next result re-establishes this in the present setting, for completeness, and it identifies a key half-space bounded by such a plane.

\begin{lemma}\label{one side} Suppose $P_1$, $P_2$, $P_3$, $P_4$ are pairwise disjoint and non-parallel planes in $\mathbb{H}^3$ such that for each $i$, a single half-space $H_i$ bounded by $P_i$ contains $P_j$ for all $j\ne i$. For any fixed $i$, there is a unique plane $\widehat{P}_i$ that intersects $P_j$ at right angles for each $j\neq i$. If $P_i$ does not meet $\widehat{P}_i$ orthogonally then there is a single half-space $\widehat{H}_i$ bounded by $\widehat{P}_i$ such that for all $j\ne i$, $\widehat{H}_i$ contains the shortest geodesic arc from $P_j$ to $P_i$.\end{lemma}

\begin{proof} For each $j\in\{1,2,3,4\}$ let $\by_j$ be an outward unit normal, in the sense described below Lemma \ref{space-like corresp}, to the half-space bounded by $P_j$ that contains each other $P_{j'}$. As observed in the proof of Lemma \ref{perp point}, for any $j\ne j'$, $\by_j$ and $\by_{j'}$ are tangent vectors to the hyperbolic geodesic intersecting $P_j$ and $P_{j'}$ perpendicularly, which has the form $\mathrm{Span}\{\by_j,\by_{j'}\}\cap\mathbb{H}^n$. By their choices they are oppositely-oriented, so $\by_j\circ\by_{j'} < 0$ by Lemma \ref{geometry of inner product}. 

For a fixed $i\in\{1,2,3,4\}$, let $\widehat{P}_i$ be the intersection with $\mathbb{H}^3$ of the span of the set of $\by_j$ for $j\ne i$. By the above, $\widehat{P}_i$ contains the mutual perpendicular geodesic to $P_j$ and $P_{j'}$, for any such distinct $j$ and $j'$, so it meets each such $P_j$ at right angles. Conversely, any plane $\widehat{P}$ that intersects any such $P_j$ at right angles contains its normal $\by_j$, since this is the tangent vector to a geodesic in $\widehat{P}$ that is perpendicular to $P_j$ at its point of intersection with $P_j$. Therefore $\widehat{P} = \widehat{P}_i$.

For each $j\ne i$, let $\bv_j$ be the point of intersection between $P_j$ and the geodesic intersecting it and $P_i$ perpendicularly. For each $j$, Lemma \ref{perp point} gives:\begin{align}\label{perp point spec}
	\bv_j = - \frac{\left(\by_i\circ \by_j\right)\by_j - \by_i}{\sqrt{(\by_i\circ \by_j)^2 - 1}}. \end{align}
Note that if any such $\bv_j$ was contained in $\widehat{P}_i$ then, since both $\widehat{P}_i$ and the shortest geodesic arc from $\bv_j$ to $P_i$ intersect $P_j$ at right angles, this entire geodesic arc would be contained in $\widehat{P}_i$. But then $P_i$ would also intersect $\widehat{P}_i$ at right angles, at the other endpoint of this geodesic arc. So because $P_i$ does not intersect $\widehat{P}_i$ at right angles by hypothesis, no such $\bv_j$ is contained in $\widehat{P}_i$.

Now fix some $j\ne i$, let $\widehat{H}_i$ be the half-space bounded by $\widehat{P}_i$ that contains $\bv_j$, and let $\bz_i$ be its outward normal, as described in Lemma \ref{space-like corresp}. As noted in the first paragraph above, for any $j'\ne j,i$, $\by_j$ is a tangent vector to the geodesic meeting $P_j$ and $P_{j'}$ perpendicularly. This geodesic lies in $\widehat{P}_i$, so $\by_j$ is a tangent vector to $\widehat{P}_i$ and is therefore orthogonal to $\bz_i$. Thus by (\ref{perp point spec}):
\[ \bz_i \circ \bv_j = \frac{\bz_i\circ\by_i}{\sqrt{(\by_i\circ \by_j)^2 - 1}} \]
Since $\bv_j$ is in the interior of $\widehat{H}_i$, $\bz_i\circ\bv_j<0$ by Lemma \ref{space-like corresp}. The equation above therefore gives $\bz_i\circ\by_i < 0$ as well. But the latter quantity does not depend on $j$, so this implies that $\bz_i\circ\bv_{j'}<0$, and hence that $\bv_{j'}\in\widehat{H}_i$ for all $j'\ne i$. The Lemma now follows from the fact that the shortest geodesic arc from any $P_j$ to $P_i$ does not not cross $\widehat{P}_i$, since each intersects $P_j$ at right angles.\end{proof}

We now consider the complementary case to that of Lemma \ref{one side}, for planes $P_1,P_2,P_3,P_4$ satisfying its hypotheses: if there exists an $i$ such that $P_i$ intersects $\widehat{P}_i$, defined as in the Lemma, at right angles, then the single plane $\widehat{P} \doteq \widehat{P}_i$ intersects all four planes at right angles and thus also equals $\widehat{P}_j$ for each $j\ne i$. The definition below accommodates both cases.

\begin{definition}\label{trunch teth} Suppose $P_1$, $P_2$, $P_3$, $P_4$ are pairwise disjoint and non-parallel planes in $\mathbb{H}^3$ such that for each $i$, a single half-space $H_i$ bounded by $P_i$ contains $P_j$ for all $j\ne i$. If for some (hence all) $i$, $\widehat{P}_i$ as in Lemma \ref{one side} does not meet $P_i$ orthogonally for any $i$, the \textit{truncated tetrahedron} determined by the $P_i$ is
\[ \Delta = \left(\bigcap_{i=1}^4 H_i  \right)\cap \left(\bigcap_{i=1}^4 \widehat{H}_i \right), \]
where $\widehat{H}_i$ is the half-space supplied by Lemma \ref{one side} for each $i$.

If $\widehat{P}_i$ does meet $P_i$ orthogonally for some $i$, then taking $\widehat{P} = \widehat{P}_i$ to be the unique plane that intersects each $P_i$ at right angles, and renumbering the $P_i$ so that the perpendicular geodesic to $P_1$ and $P_3$ separates $P_2\cap\widehat{P}$ from $P_4\cap\widehat{P}$, we define $\Delta$ as a \textit{degenerate truncated tetrahedron}  by:
\[ \Delta = \widehat{P}\cap \left(\bigcap_{i=1}^4 H_i  \right)\cap h_{12} \cap h_{23}\cap h_{34}\cap h_{14}, \] 
where $h_{12}$ is the half-plane in $\widehat{P}$ that is bounded by the perpendicular geodesic to $P_1$ and $P_2$ that contains $P_4\cap\widehat{P}$ (hence also $P_3\cap \widehat{P}$); and so on.

For each $i<j\le 4$, denote the shortest arc in $\mathbb{H}^3$ joining $P_i$ to $P_j$ as $\lambda_{ij}$ and call it an \textit{internal edge} of $\Delta$. The internal edge \textit{opposite} $\lambda_{ij}$ is $\lambda_{kl}$, where $k<l\in\{1,2,3,4\}-\{i,j\}$. For each $i$, the  \textit{internal face opposite $P_i$} is the right-angled hexagon in $\Delta\cap\widehat{P}_i$ bounded by the internal edges $\lambda_{jk}$, for each pair $j<k\in\{1,2,3,4\}-\{i\}$, and arcs of the $P_j$, $j\ne i$. The non-internal faces and edges of $\Delta$ are \textit{external}. Each of these is entirely contained in $P_i$ for some $i$. 

The \textit{transversal} of $\Delta$ joining an internal edge $\lambda_{ij}$ to its opposite $\lambda_{kl}$ is the shortest geodesic arc with one endpoint on each edge; or if these edges intersect, it is their point of intersection.\end{definition}

Note that if for some $i< j$, $\lambda_{ij}$ intersects its opposite edge $\lambda_{kl}$, then $\Delta$ is degenerate since the plane containing both $\lambda_{ij}$ and $\lambda_{kl}$ intersects all four $P_i$ orthogonally. Conversely, if $\Delta$ is degenerate then it is a right-angled octagon in $\widehat{P}$, and with the $P_i$ numbered as in this case of Definition \ref{trunch teth}, the opposite edges $\lambda_{13}$ and $\lambda_{24}$ do intersect.

In the non-degenerate case, each internal face of $\Delta$ is of the form $\Delta\cap\widehat{P}_i$ for some $i$, and each edge $\lambda_{ij}$ is the intersection of the internal faces contained in $\widehat{P}_k$ and $\widehat{P}_{l}$ for $k<l\in\{1,2,3,4\}-\{i,j\}$. In this case, $\Delta$ is homeomorphic to the complement in a tetrahedron of the union of small regular neighborhoods of the vertices; see Figure \ref{trunc tet}. 

\begin{figure}

\begin{tikzpicture}[scale=0.75]

\draw [dashed] (0,0) -- (1.5,2) -- (3,0) -- (2,-1) -- cycle;
\draw [dashed] (1.5,2) -- (2,-1);
\draw [dashed] (0,0) -- (3,0);

\fill [color=white] (0.24,0) circle [radius=0.1];
\fill [color=white] (2.56,0) circle [radius=0.1];

\draw [very thick, dashed] (0.24,-0.12) -- (0.44,0) -- (2.56,0) -- (2.76,-0.24);
\draw [very thick, dashed] (0.44,0) -- (0.24,0.32);
\draw [very thick, dashed] (2.56,0) -- (2.76,0.32);
\fill [color=white] (1.83,0) circle [radius=0.1];
\draw [very thick] (1.2,1.6) -- (1.8,1.6) -- (2.76,0.32) -- (2.76,-0.24) -- (2.2,-0.8) -- (1.6,-0.8) -- (0.24,-0.12) -- (0.24,0.32) -- cycle;
\draw [very thick] (1.2,1.6) -- (1.59,1.46) -- (1.8,1.6);
\draw [very thick] (1.6,-0.8) -- (1.93,-0.58) -- (2.2,-0.8);
\draw [very thick] (1.59,1.46) -- (1.93,-0.58);

\node [right] at (1.5,2) {\small $\by_1$};
\node [right] at (3,0) {\small $\by_4$};
\node [left] at (0,0) {\small $\by_2$};
\node [left] at (2,-1) {\small $\by_3$};

\fill (1.2,1.6) circle [radius=0.07];
\node [left] at (1.2,1.6) {\scriptsize $\bv_1$};
\fill (2.2,-0.8) circle [radius=0.07];
\node [right] at (2.2,-0.8) {\scriptsize $\bv_3$};

\end{tikzpicture}

\caption{A truncated tetrahedron, with missing ``vertices'' labeled by space-like vectors. Notation as in the proof of Lemma \ref{shrinking transversals}.}
\label{trunc tet}
\end{figure}
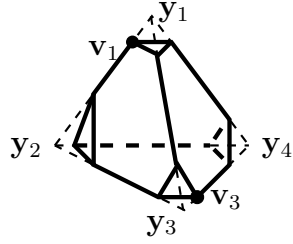

The main results of this section record some observations about the lengths of transversals of truncated tetrahedra. Before embarking on this we record the following basic calculus fact.

\begin{lemma}\label{or test sharc} Any positive linear combination $D(s,t) = C_{++} e^se^t + C_{+-}e^s e^{-t} + C_{-+} e^{-s}e^t + C_{--} e^{-s}e^{-t}$ of the functions $e^se^t$, $e^se^{-t}$, $e^{-s}e^t$, and $e^{-s}e^{-t}$ is strictly convex and has a unique critical point $(s_0,t_0)$, at which it attains an absolute minimum. It satisfies:
\begin{align}\label{crit loc}
	& e^{s_0} = \left(\frac{C_{-+}C_{--}}{C_{+-}C_{++}}\right)^{1/4}, & & e^{t_0} = \left( \frac{C_{+-}C_{--}}{C_{-+}C_{++}} \right)^{1/4}, && \frac{D(s_0,t_0)}{2} = \sqrt{C_{++}C_{--}} + \sqrt{C_{+-}C_{-+}}.
\end{align}
Furthermore, if $0 < C_{++} < C_{+-}, C_{-+} < C_{--}$ then $(s_0,t_0)\in (0,\infty)^2$.\end{lemma}

\begin{proof} The Hessian of each of $e^se^t$ and $e^{-s}e^{-t}$ is a scalar multiple of $\left(\begin{smallmatrix} 1 & 1 \\ 1 & 1 \end{smallmatrix}\right)$, and that of each of $e^se^{-t}$ and $e^{-s}e^t$ is a multiple of $\left(\begin{smallmatrix} 1 & -1 \\ -1 & 1 \end{smallmatrix}\right)$. Each of these matrices has $0$ and $2$ as eigenvalues and $\left(\begin{smallmatrix} 1 \\ 1 \end{smallmatrix}\right)$ and $\left(\begin{smallmatrix} 1 \\ -1 \end{smallmatrix}\right)$ as eigenvectors, but the null eigenvector of either is associated to the eigenvalue $2$ of the other. It follows that $D(s,t)$, being a positive linear combination of the four functions, has positive-definite Hessian and therefore is strictly convex.

We note that $D(s,t)\to\infty$ as $(s,t)\to\infty$. One can see this using the first summand $C_{++}e^se^t$ for $(s,t)$ in the first quadrant; the summand $C_{-+}e^{-s}e^t$ for $(s,t)$ in the second; $C_{--}e^{-s}e^{-t}$ for $(s,t)$ in the third; and $C_{+-}e^se^{-t}$ for $(s,t)$ in the fourth. $D(s,t)$ therefore attains an absolute minimum at a critical point, which is unique by convexity. To identify the critical point of $D(s,t)$ as defined above, we compute:\begin{align*}
	\partial_s D(s,t) & = C_{++} e^se^t + C_{+-}e^se^{-t} - C_{-+}e^{-s}e^t - C_{--}e^{-s}e^{-t} \\
	\partial_t D(s,t) & = C_{++} e^se^t - C_{+-}e^se^{-t} + C_{-+}e^{-s}e^t - C_{--}e^{-s}e^{-t}. \end{align*}
Setting each of these equal to $0$, subtracting bottom from top, and multiplying by $e^{s_0}e^{t_0}/2$ yields $0 = C_{+-} e^{2s_0} - C_{-+} e^{2t_0}$, giving $e^{2s_0} = e^{2t_0}C_{-+}/C_{+-}$ for the critical point $(s_0,t_0)$. Adding the two equations and multiplying through by $e^{s_0}e^{t_0}/2$ yields $0 = C_{++} e^{2s_0}e^{2t_0} - C_{--}$. After substituting for $e^{2s_0}$ we solve for $e^{2t_0}$, then plug the result back in to obtain the formulas of (\ref{crit loc}). 

The sequence of inequalities in this result's statement that starts with $0 < C_{++}$ implies that $C_{--}/C_{++}$ is greater than both $C_{+-}/C_{-+}$ and its reciprocal $C_{-+}/C_{+-}$, which in turn implies that the quantities in parenthesis in (\ref{crit loc}) are each greater than $1$. Therefore if these inequalities hold then each of $s_0$ and $t_0$ is greater than $0$. 
\end{proof}

\begin{lemma}\label{shrinking transversals} Suppose $P_1$, $P_2$, $P_3$, $P_4$ are pairwise disjoint and non-parallel planes in $\mathbb{H}^3$ such that for each $i$, a single half-space $H_i$ bounded by $P_i$ contains $P_j$ for all $j\ne i$, and let $\Delta$ be the truncated tetrahedron determined by the $P_i$ as in Definition \ref{trunch teth}. For any fixed $i<j\in\{1,2,3,4\}$ and $k<l\in\{1,2,3,4\}-\{i,j\}$, let $\tilde\lambda_{ij}$ and $\tilde\lambda_{kl}$ be the geodesics respectively intersecting $P_i$ to $P_j$, and $P_k$ and $P_l$, at right angles. Fixing parametrizations $\tilde\lambda_{ij}(s)$ and $\tilde\lambda_{kl}(t)$ by arclength:\begin{enumerate}
	\item the function $D(s,t)$ that records the hyperbolic cosine of the distance from $\tilde\lambda_{ij}(s)$ to $\tilde\lambda_{kl}(t)$ has a unique critical point $(s_0,t_0)$ in $\mathbb{R}^2$, at which it attains an absolute minimum;
	\item the absolute minimum value $D(s_0,t_0)$ depends only on the pairwise distances $\ell_{i'j'}$ between the $P_{i'j'}$, for $i'\ne j'\in\{1,2,3,4\}$;
	\item the points $\tilde\lambda_{ij}(s_0)$ and $\tilde\lambda_{kl}(t_0)$ lie in $\Delta$, so $D(s_0,t_0)$ is the hyperbolic cosine of a transversal length of $\Delta$; and
	\item written as $T(x,y;a,b,c,d)$, where $x = \cosh\ell(\lambda_{ij})$, $y = \cosh \ell(\lambda_{kl})$, and $a$, $b$, $c$, $d$ are hyperbolic cosines of the other four internal edge lengths of $\Delta$, this transversal length is invariant under even involutions of $\{a,b,c,d\}$, and strictly increasing in each of these variables, for any fixed $x,y>1$.
\end{enumerate}\end{lemma}

Convexity of the distance between two parametrized geodesics is well-known in broad generality (see eg.~\cite[Prop.~2.2]{BridsonHaefl} for the CAT(0) context). Conclusion (3) above is implied by \cite[Prop.~2.7]{FrigePet}. However our unified proof of all conclusions, which follows the general perspective taken in this note, will set us up to prove this section's main result.

\begin{proof} Similarly to the proof of Lemma \ref{one side}, for each $i\in\{1,2,3,4\}$ let $\by_i$ be an outward unit normal, in the sense described below Lemma \ref{space-like corresp}, to the half-space $H_i$ bounded by $P_i$ that contains each other $P_j$. For any $j\ne i$, $\by_i$ and $\by_{j}$ are then oppositely-oriented tangent vectors to the hyperbolic geodesic intersecting $P_i$ and $P_{j}$ perpendicularly, so $\by_i\circ\by_{j} < 0$ by Lemma \ref{geometry of inner product}. Furthermore, by this result the hyperbolic cosine of the distance from $P_i$ to $P_j$, which we will here denote $L_{ij}$, satisfies $L_{ij} = -\by_i\circ\by_j$ for each $i<j\in\{1,2,3,4\}$. Since the $P_i$ are labeled arbitrarily, we may take $(i,j,k,l) = (1,2,3,4)$ without loss of generality. 

Let $\bv_1=\tilde\lambda_{12}(0)$ and $\bv_3= \tilde\lambda_{34}(0)$ be the points of intersection $\tilde\lambda_{12}\cap P_1$ and $\tilde\lambda_{34}\cap P_3$, respectively. Then $\tilde\lambda_{12}$ is parametrized by arclength as $\tilde\lambda_{12}(s) = \cosh s\,\bv_1 - \sinh s\, \by_1$ starting at $\bv_1$ and running into $H_1$ using (\ref{arbitrary}), since $\by_1$ is outward-pointing from $H_1$, and likewise $\tilde\lambda_{34}(t) = \cosh t\,\bv_3 - \sinh t\,\by_3$ starts at $\bv_3$ and runs into $H_3$. $D(s,t)$ described above thus satisfies
\begin{align} D(s,t) & = - (\cosh s\,\bv_1 - \sinh s\, \by_1)\circ(\cosh t\,\bv_3 - \sinh t\,\by_3) \nonumber\\
	& = -\cosh s\cosh t\, (\bv_1\circ\bv_3) + \cosh s\sinh t\, (\bv_1\circ\by_3) \label{suzie schmuckaroozie}\\
	& \qquad + \sinh s \cosh t\, (\by_1\circ\bv_3) - \sinh s\sinh t\, (\by_1\circ\by_3) \nonumber \\
	& = C_{++} e^se^t + C_{+-}e^se^{-t} + C_{-+}e^{-s}e^t + C_{--}e^{-s}e^{-t}, \nonumber \end{align}
where:\begin{align}
	C_{++} & = \frac{1}{4}[-\bv_1\circ\bv_3 + \bv_1\circ\by_3 + \by_1\circ\bv_3 - \by_1\circ\by_3] \label{C plus plus} \\
	C_{+-} & = \frac{1}{4}[-\bv_1\circ\bv_3 - \bv_1\circ\by_3 + \by_1\circ\bv_3 + \by_1\circ\by_3] \nonumber \\
	C_{-+} & = \frac{1}{4}[-\bv_1\circ\bv_3 + \bv_1\circ\by_3 - \by_1\circ\bv_3 + \by_1\circ\by_3] \nonumber \\
	C_{--} & = \frac{1}{4}[-\bv_1\circ\bv_3 - \bv_1\circ\by_3 - \by_1\circ\bv_3 - \by_1\circ\by_3] \nonumber \end{align}
We now record each inner product above in terms of the distances $L_{ij}$ from $P_i$ to $P_j$ by substituting the formulas for $\bv_1$ and $\bv_3$ from Lemma \ref{perp point}.
\begin{align}
	& \bv_1\circ\bv_3 = - \frac{L_{12}L_{13}L_{34} + L_{23}L_{34} + L_{12}L_{14} + L_{24}}{\sqrt{(L_{12}^2 - 1)(L_{34}^2-1)}}, \label{blib} \\
	& \bv_1\circ\by_3 = - \frac{L_{12}L_{13}+L_{23}}{\sqrt{L_{12}^2 - 1}}, \quad 
		\by_1\circ\bv_3 = - \frac{L_{13}L_{34}+L_{14}}{\sqrt{L_{34}^2-1}}, \quad \by_1\circ\by_3 = -L_{13} \nonumber \end{align}
We claim that $0< C_{++} < C_{+-}, C_{-+} < C_{--}$. We first consider $C_{++}$. Plugging in for the terms of (\ref{C plus plus}) and rearranging yields:\begin{align*}
	\frac{1}{4} & \left[  L_{13}\left( \frac{L_{12}L_{34}}{\sqrt{(L_{12}^2-1)(L_{34}^2-1)}} - \frac{L_{12}}{\sqrt{L_{12}^2-1}} - \frac{L_{34}}{\sqrt{L_{34}^2-1}} + 1\right) \right. \\
		& \quad \left. + \frac{L_{23}}{\sqrt{L_{12}^2-1}} \left(\frac{L_{34}}{\sqrt{L_{34}^2-1}}-1\right) + \frac{L_{14}}{\sqrt{L_{34}^2-1}} \left(\frac{L_{12}}{\sqrt{L_{12}^2-1}}-1\right) + 
			\frac{L_{24}}{\sqrt{(L_{12}^2-1)(L_{34}^2-1)}} \right] \end{align*}
Here, bearing in mind that each $L_{ij}$ is greater than one since it is the hyperbolic cosine of a positive length, we conclude that each term in parentheses is positive, hence that $C_{++}>0$. Now:
\[ C_{-+} - C_{++} = -2\bv_1\circ\by_3 + 2\by_1\circ\by_3 = 2L_{13}\left(\frac{L_{12}}{\sqrt{L_{12}^2 - 1}} - 1 \right) + \frac{2L_{23}}{\sqrt{L_{12}^2-1}} > 0. \]
A similar argument shows that $C_{+-} > C_{++}$. And since each of the four summands of $C_{--}$ has positive sign, it is greater than the other coefficients, proving the claim.

Given the claim, Lemma \ref{or test sharc} implies assertion (1) of this result. The fact that the coefficients $C_{\pm\pm}$ depend only on the $L_{ij}$ implies the same for the coordinates of the critical point $(s_0,t_0)$ of $D$, and hence for the absolute minimum value $D(s_0,t_0)$. Lemma \ref{or test sharc} also gives that $(s_0,t_0)\in(0,\infty)^2$; because $\tilde\lambda_{ij}$ was parametrized pointing into $H_i$ with $\tilde\lambda_{ij}(0)\in P_i$, the fact that $s_0>0$ implies that the closest point $\tilde\lambda_{ij}(s_0)$ to $\tilde\lambda_{kl}$ lies in $H_i$. Likewise, $\tilde\lambda_{kl}(t_0)\in H_k$. But by swapping the roles of $i$ and $j$, and of $k$ and $l$---which only re-parametrizes $\tilde\lambda_{12}$ and/or $\tilde\lambda_{34}$---and re-running the argument above, we find that the closest point of $\tilde\lambda_{ij}$ to $\tilde\lambda_{kl}$ also lies in $H_j$. Therefore it lies in the edge $\lambda_{ij}$ of $\Delta$. Likewise, the closest point of $\tilde\lambda_{kl}$ to $\tilde\lambda_{ij}$ lies in the edge $\lambda_{kl}$. Therefore the geodesic arc joining these two points is the relevant transversal of $\Delta$, as asserted in (3) above.

The transversal's length $T(x,y;a,b,c,d)$ is therefore the arccosh of the minimum value of $D(s,t)$, so the fact that it depends only on the internal edge lengths follows directly from descriptions of $D$ and of $(s_0,t_0)$ in terms of the $C_{\pm\pm}$---these depend only on the internal edge lengths ie.~the pairwise distances between planes. The symmetry property of $T$ follows again from the fact that swapping the indicex of $P_i$ with $P_j$, and/or of $P_k$ with $P_l$, does not change its value. Each such swap acts on the collection $\{a,b,c,d\}$ as an even involution.

It remains to show that $T(x,y;a,b,c,d)$ is increasing in each of its last four variables. For this (again taking $i = 1$, $j = 2$, $k = 3$, and $l = 4$ after renumbering the $P_i$ if necessary), we claim that for any fixed $(s,t)$ the value $D(s,t)$---with itself depends on the values $x = L_{12}$, $y = L_{34}$, $a = L_{13}$, $b = L_{14}$, $c = L_{23}$, and $d = L_{24}$---increases with $d=L_{24}$ (taking the other $L_{ij}$ as fixed). To see this, we plug the values from (\ref{blib}) into the formula of (\ref{suzie schmuckaroozie}). Dividing out by $\cosh s\cosh t$ yields:\begin{align*}
  \frac{D(s,t)}{\cosh s\cosh t} & = \frac{L_{12}L_{13}L_{34} + L_{12}L_{14} + L_{23}L_{34} + L_{24}}{\sqrt{(L_{12}^2-1)~(L_{34}^2-1)}}  \\
  &\qquad - \tanh t \frac{L_{12}L_{13}+L_{23}}{\sqrt{L_{12}^2-1}} - \tanh s \frac{L_{13}L_{34} + L_{14}}{\sqrt{L_{34}^2-1}} + \tanh s \tanh t\ L_{13}.  \end{align*}
The value $d = L_{24}$ appears only once in this formula, in the numerator of the first summand above, with a positive sign. The claim is thus clear. It follows that the absolute minimum of $D(s,t)$, and hence also the transversal length $T(x,y;a,b,c,d)$, also increases with $d$. Since values of $T(x,y;a,b,c,d)$ are invariant under a transitive group action on $\{a, b,c,d\}$, the same then holds for $a$, $b$ and $c$.\end{proof}

\begin{proposition}\label{tetra transversals} Suppose $P_1$, $P_2$, $P_3$, $P_4$ are pairwise disjoint and non-parallel planes in $\mathbb{H}^3$ such that for each $i$, a single half-space $H_i$ bounded by $P_i$ contains $P_j$ for all $j\ne i$, and let $\Delta$ be the truncated tetrahedron determined by the $P_i$ as in Definition \ref{trunch teth}. For any fixed $i<j\in\{1,2,3,4\}$ and $k<l\in\{1,2,3,4\}-\{i,j\}$, let $T(x,y;a,b,c,d)$ record the length of the transversal of $\Delta$ joining $\lambda_{ij}$ to $\lambda_{kl}$ as in Lemma \ref{shrinking transversals}, where $x = \cosh\ell(\lambda_{ij})$, $y = \cosh \ell(\lambda_{kl})$, and $a$, $b$, $c$, $d$ are hyperbolic cosines of the other four internal edge lengths of $\Delta$. For some fixed $L>1$, if each of $a$, $b$, $c$, and $d$ is at least $L$, then
\[ \cosh T(x,y;a,b,c,d) \ge \frac{2L}{\sqrt{(x-1)(y-1)}}, \]
with equality if and only if $a = b = c = d = L$.\end{proposition}

The Proposition's proof rests on the relative convenience of explicitly writing down the critical point of the function $D(s,t)$ from Lemma \ref{or test sharc}, and hence explictly computing values of $T$, in this highly symmetric case.

\begin{proof} We re-record the formulas of (\ref{blib}) in the special case that $L_{13} = L_{14} = L_{23} = L_{24} = L$, and $L_{12} = x$, $L_{34} = y$:\begin{align*}
	& \bv_1\circ\bv_3 = - L\sqrt{\frac{(x+1)(y+1)}{(x-1)(y-1)}}, & 
	& \bv_1\circ\by_3 = - L\sqrt{\frac{x+1}{x-1}}, & 
	& \by_1\circ\bv_3 = - L\sqrt{\frac{y+1}{y-1}}, & & \by_1\circ\by_3 = -L \end{align*}
Plugging these into the formula (\ref{crit loc}) from Lemma \ref{or test sharc}, we find that $\cosh s_0 = \sqrt{\frac{1}{2}(x+1)}$ and $\cosh t_0 = \sqrt{\frac{1}{2}(y+1)}$. (This means that $s_0 = \ell(\lambda_{12})/2$ and $t_0 = \ell(\lambda_{34})/2$, which one would expect from considerations of symmetry.) Plugging this point into the formula for $D$ from the Lemma yields:
\begin{align*} D(s_0,t_0) & = -\cosh s_0\cosh t_0\, (\bv_1\circ\bv_3) + \cosh s_0\sinh t_0\, (\bv_1\circ\by_3) \\
	& \qquad + \sinh s_0 \cosh t_0\, (\by_1\circ\bv_3) - \sinh s_0\sinh t_0\, (\by_1\circ\by_3) \\
	& = \frac{L}{2} \left[ \frac{(x+1)(y+1)}{\sqrt{(x-1)(y-1)}} - \frac{(x+1)\sqrt{y-1}}{\sqrt{x-1}} - \frac{\sqrt{x-1}(y+1)}{\sqrt{y-1}} + \sqrt{(x-1)(y-1)} \right] \end{align*}
When simplified, this yields the formula of the Proposition statement. It now follows from Lemma \ref{shrinking transversals} that this bounds the value of $\cosh T(x,y;a,b,c,d)$ below when $a$, $b$, $c$, and $d$ are all at least $L$, and that equality holds if and only if $a = b = c = d = L$.
\end{proof}

\section{Dimension three: \emph{partially} truncated tetrahedra}\label{PT}

We now consider cases in which, for a fixed $k\in\{1,2,3\}$, the hyperplane $P_i$ of the previous section is replaced by a horoball $B_i$ for each $i>k$. We again assume that the planes $P_1,\hdots,P_k$ are pairwise disjoint and non-parallel; require the ideal points of $B_{k+1},\hdots,B_4$ to be pairwise distinct and not contained in the ideal boundary of any $P_i$; and for each $i\le k$, assume that a single half-space $H_i$ bounded by $P_i$ contains all $P_j$, $j\ne i$, and (in its ideal boundary) the ideal point of each $B_{j'}$. Our first result is the analog of Lemma \ref{one side} in this setting.

\begin{lemma}\label{one side prime} For a fixed $k\in\{1,2,3\}$, suppose $P_1,\hdots,P_k$ are pairwise disjoint and non-parallel planes, and $B_{k+1},\hdots, B_4$ are horoballs defined by positive light-like vectors $\bx_{k+1},\hdots,\bx_4$ such that the ideal points $[\bx_j]$ are pairwise distinct and each not contained in the ideal boundary of any $P_i$. Further, suppose for each $i\le k$ that there is a single half-space $H_i$ bounded by $P_i$ which contains $P_j$ for all $j\ne i\le k$, and whose ideal boundary contains each ideal point $[\bx_{j'}]$. 

For any fixed $i$, there is a unique plane $\widehat{P}_i$ that intersects $P_j$ at right angles for each $j\neq i\le k$, and whose ideal boundary contains each ideal point $[\bx_{j'}]$ for $j'\ne i > k$ (whence $\widehat{P}_i$ also intersects $B_{j'}$ at a right angle).\begin{enumerate}
	\item For $i\le k$, if $\widehat{P}_i$ does not meet $P_i$ orthogonally then there is a single half-space $\widehat{H}_i$ bounded by $\widehat{P}_i$ such that  $\widehat{H}_i$ contains the shortest geodesic arc from $P_j$ to $P_i$ for each $j\ne i\le k$, and $\widehat{H}_i$ contains the entire geodesic ray from $P_i$ in the direction of $\bx_{j'}$ for each $j'>k$.
	\item For $i>k$, if the ideal boundary of $\widehat{P}_i$ does not contain $[\bx_i]$ then there is a single half-space $\widehat{H}_i$ bounded by $\widehat{P}_i$ such that for each $j\le k$, $\widehat{H}_i$ contains the entire geodesic perpendicular to $P_j$ in the direction of $\bx_i$, and for each $j'\ne i >k$, the entire geodesic joining $\bx_i$ to $\bx_{j'}$.
\end{enumerate}\end{lemma}

\begin{proof} Let $k\in\{1,2,3\}$ be as in the Lemma's statement. For each $i\in\{1,\hdots,k\}$, as in the proof of Lemma \ref{one side} let $\by_i$ be an outward unit normal, in the sense described below Lemma \ref{space-like corresp}, to the half-space bounded by $P_i$ that contains each $P_{j}$ or $B_j$ for $j\ne i$. As in the proof of Lemma \ref{one side}, for any distinct $j, j'\in\{1,\hdots,k\}$, $\by_j\circ\by_{j'} < 0$. For $j\le k$ and $j'>k$, $\by_j\circ\bx_{j'} <0$ by Lemma \ref{to the hyperplane!} and the choice of $\by_j$. For distinct $j,j'>k$, we note that $\bx_j$ and $\bx_{j'}$ are linearly independent since their projective classes are distinct, so $\bx_j\circ\bx_{j'}<0$ by Fact \ref{CS}. For any fixed $i\in\{1,2,3,4\}$, let $\widehat{P}_i$ be the intersection with $\mathbb{H}^3$ of the span of the set of $\by_j$ or $\bx_j$ as above, taken over the three $j\ne i\in \{1,2,3,4\}$. 

As in the proof of Lemma \ref{one side}, $\widehat{P}_i$ contains the mutual perpendicular geodesic to $P_j$ and $P_{j'}$, for any distinct $j$ and $j'\le k$ and not $i$, so it meets each such $P_j$ at right angles. It contains the ideal points $[\bx_{j}]$, $j\ne i > k$, by construction. Moreover, for $j,j'\ne i$, if $j\le k$ and $j'>k$ then $\widehat{P}_i$ contains the geodesic perpendicular to $P_j$ in the direction of $\bx_{j'}$ described in Lemma \ref{to the hyperplane!}; and if $j,j'>k$ then $\widehat{P}_i$ contains the geodesic between $\bx_j$ and $\bx_{j'}$ described in Lemma \ref{to the other horoball!}. As Remarked below Lemma \ref{to the horoball!}, this implies that $\widehat{P}_i$ intersects each $B_j$ perpendicularly.

We now suppose that $i\le k$ and that $\widehat{P}_i$ does not intersect $P_i$ at a right angle. For any $j\ne i\le k$, arguing as in Lemma \ref{one side} we find that the point $\bv_j\in P_j$ at the foot of the shortest geodesic joining $P_j$ to $P_i$ does not lie in $\widehat{P}_i$, and that for the outward unit normal $\bz_i$ to the half-space $\widehat{H}_i$ bounded by $\widehat{P}_i$ that contains $\bv_j$, $\by_i\circ\bz_i<0$. For any $j'>k$, the geodesic perpendicular to $P_i$ in the direction of $\bx_{j'}$ is described by Lemma \ref{to the hyperplane!} as $\gamma(t) = e^{-h}\cosh t\,\bx_{j'} + e^{-t}\by_i$, where $e^h = -\bx_{j'}\circ\by_i$. Since $\bx_{j'}\circ\bz_i = 0$ by the choice of $\bz_i$, $\gamma(t)\circ\bz_i = e^{-t}\by_i\circ\bz_i < 0$ for any $t$. Therefore $\gamma(t)$ lies in $\widehat{H}_i$. This establishes the Lemma's conclusion (1).

Now take $i>k$ and suppose that the ideal boundary of $\widehat{P}_i$ does not contain $[\bx_i]$. For any $j\leq k$, the geodesic perpendicular to $P_j$ and in the direction of $\bx_i$ is described as $\gamma_j(t) =  e^{-h}\cosh t\,\bx_{i} + e^{-t}\by_j$ by Lemma \ref{to the hyperplane!}, where $e^h = -\bx_i\circ\by_j$. For a unit normal $\bz_i$ to $\widehat{P}_i$ and any $t\in\mathbb{R}$, $\gamma_j(t)\circ\bz_i = e^{-h}\cosh t\, \bx_i\circ\bz_i$ has the same sign as $\bx_i\circ\bz_i$. Choosing $\bz_i$ to be the outward unit normal to the half-space $\widehat{H}_i$ containing $\bx_i$ in its ideal boundary, it follows that $\gamma_j(t)$ is contained in $\widehat{H}_i$ for all $t\in\mathbb{R}$. For any $j'\ne i>k$, the geodesic joining $\bx_i$ to $\bx_{j'}$ has the form $\gamma_{j'}(t) = \frac{1}{2}e^{-d/2}\left(e^t\,\bx_i + e^{-t}\,\bx_{j'}\right)$ by Lemma \ref{to the other horoball!}, where $e^d = -\frac{1}{2}\bx_i\circ\bx_{j'}$. Thus for $\bz_i$ as above, since $\bx_{j'}$ is in the ideal boundary of $\widehat{P}_i$, $\gamma_{j'}(t)\circ\bz_i = \frac{1}{2}e^{-d/2}e^t\,\bx_i\circ\bz_i < 0$. Therefore again $\gamma_{j'}(t)\in\widehat{H}_i$ for all $t$. This establishes the Lemma's conclusion (2).
\end{proof}

Again if there is any $i$ such that either $\widehat{P}_i$ intersects $P_i$ at right angles (if $i\le k$) or contains $\bx_i$ in its ideal boundary ($i>k$), then in fact $\widehat{P}_j = \widehat{P}_i$ has the same property for all $j$, and the definition below produces a degenerate tetrahedron that lies entirely in this single plane.

\begin{definition}\label{french tuth} For a fixed $k\in\{1,2,3\}$, suppose $P_1,\hdots,P_k$ are pairwise disjoint and non-parallel planes, and $B_{k+1},\hdots, B_4$ are horoballs defined by positive light-like vectors $\bx_{k+1},\hdots,\bx_4$ such that the ideal points $[\bx_j]$ are pairwise distinct and each not contained in the ideal boundary of any $P_i$. Further, suppose for each $i\le k$ that there is a single half-space $H_i$ bounded by $P_i$ which contains $P_j$ for all $j\ne i\le k$, and whose ideal boundary contains each ideal point $[\bx_{j'}]$. 

If $\widehat{P}_1$ as in Lemma \ref{one side prime} does not meet $P_1$ orthogonally, taking $\widehat{H}_i$ as the half-space supplied by Lemma \ref{one side prime} for each $i$, define the \textit{partially truncated tetrahedron} determined by the $P_i$ and $\bx_j$ as
\[ \Delta = \left(\bigcap_{i=1}^k H_i  \right)\cap \left(\bigcap_{i=1}^4 \widehat{H}_i \right). \]

If $\widehat{P}_1$ does meet $P_1$ orthogonally, then taking $\widehat{P} = \widehat{P}_i$ to be the unique plane that intersects each $P_i$ at right angles for $i\le k$, and contains $\bx_j$ for each $j>k$; and renumbering the $P_i$ so that the perpendicular geodesic to $P_1$ and $P_3$ (if $k=3$) or to $P_1$ in the direction of $\bx_3$ (if $k<3$) separates $P_2\cap\widehat{P}$ or $\bx_2$ from $\bx_4$, we define $\Delta$ as a \textit{degenerate partially truncated tetrahedron}  by:
\[ \Delta = \widehat{P}\cap \left(\bigcap_{i=1}^k H_i  \right)\cap h_{12} \cap h_{23}\cap h_{34}\cap h_{14}, \] 
where $h_{12}$ is the half-plane in $\widehat{P}$ that is bounded by the perpendicular geodesic to $P_1$ and $P_2$---or, if $k=1$, in the direction of $\bx_2$---and contains  $\bx_4$ in its ideal boundary (hence also contains $P_3\cap \widehat{P}$, or $\bx_3$ in its ideal boundary); and so on.

In all cases, for each $i>k$, we say the projective class $[\bx_i]$ of $\bx_i$ is an \textit{ideal vertex} of $\widehat{P}_i$. For each $i<j\le k$, the \textit{internal edge} $\lambda_{ij}$ of $\Delta$ is the shortest arc in $\mathbb{H}^3$ joining $P_i$ to $P_j$; for $i\le k<j$, the \textit{ray edge} $\lambda_{ij}$ is the geodesic ray perpendicular to $P_i$ in the direction of $\bx_j$; and for $ k < i< j$, the \textit{bi-infinite edge} $\lambda_{ij}$ is the geodesic joining $\bx_i$ to $\bx_j$. The edge \textit{opposite} $\lambda_{ij}$ is $\lambda_{i'j'}$, where $i'<j'\in\{1,2,3,4\}-\{i,j\}$. For each $i$, the  \textit{internal face opposite} $P_i$ or $\bx_i$ is the polygon in $\Delta\cap\widehat{P}_i$ bounded by the internal edges $\lambda_{jj'}$, for each pair $j<j'\in\{1,2,3,4\}-\{i\}$, and arcs of the $P_j$, $j\ne i\le k$. The non-internal faces and edges of $\Delta$ are \textit{external}. Each of these is entirely contained in $P_i$ for some $i\le k$. 

The \textit{transversal} of $\Delta$ joining an edge $\lambda_{ij}$ to its opposite $\lambda_{kl}$ is the shortest geodesic arc with one endpoint on each edge; or if these edges intersect, it is their point of intersection.\end{definition}

The rest of the section is divided into subsections addressing the different possible $k\in\{1,2,3\}$.

\subsection{One horoball} Here we take $k = 1$, the case of partially truncated tetrahedra having a single ideal vertex. We first give the analog of Lemma \ref{shrinking transversals} in this case.

\begin{lemma}\label{shortest arc one} Suppose $P_1$, $P_2$, and $P_3$ are pairwise disjoint and non-parallel planes in $\mathbb{H}^3$, and $B_4$ is a horoball defined by a vector $\bx_4$ not in the ideal boundary of any $P_i$, such that for each $i<4$, a single half-space $H_i$ bounded by $P_i$ contains $P_j$ for all $j\ne i<4$ and $\bx_4$ in its ideal boundary. Let $\Delta$ be the partially truncated tetrahedron determined by the $P_i$ and $B_4$ as in Definition \ref{one side prime}.  For any $i< j\in\{1,2,3\}$, let $\tilde\lambda_{ij}$ be the geodesic intersecting $P_i$ and $P_j$ at right angles; and for $i<4$ let $\tilde\lambda_{i4}$ be the perpendicular geodesic to $P_i$ in the direction of $\bx_4$. Fixing such an $i<j\le 3$ and $k\in\{1,2,3\}-\{i,j\}$, and fixing parametrizations $\tilde\lambda_{ij}(s)$ and $\tilde\lambda_{k4}(t)$ by arclength:\begin{enumerate} 
	\item the function $D(s,t)$ that records the hyperbolic cosine of the distance from $\tilde\lambda_{ij}(s)$ to $\tilde\lambda_{k4}(t)$ has a unique critical point $(s_0,t_0)$ in $\mathbb{R}^2$, at which it attains an absolute minimum;
	\item the absolute minimum value $D(s_0,t_0)$ depends only on the pairwise distances $\ell_{i'j'}$ between the $P_{ij}$ and signed distances $h_{k'4}$ from the $P_i$ to $B_4$; 
	\item the points $\tilde\lambda_{ij}(s_0)$ and $\tilde\lambda_{k4}(t_0)$ lie in $\Delta$, so $D(s_0,t_0)$ is the hyperbolic cosine of a transversal length of $\Delta$; and
	\item written as $T_1(x,y;a,b;c,d)$, where $x = \cosh \ell_{ij}$, $y = e^{h_{k4}}$, $a$ and $b$ are hyperbolic cosines of the other internal edge lengths, and $c$ and $d$ are of the form $e^{h_{i4}}$ for the signed distance between $B_4$ and the external faces in $P_i$, $i\ne k$, its value is invariant under the involution that swaps its inputs $a$ with $b$ and $c$ with $d$, and increasing in each of $a$, $b$, $c$, $d$ individually.
\end{enumerate}\end{lemma}

\begin{proof} For each $i\in\{1,2,3\}$ let $\by_i$ be an outward unit normal, in the sense described below Lemma \ref{space-like corresp}, to the half-space $H_i$ bounded by $P_i$ that contains each other $P_j$. For any $j\ne i$, $\by_i$ and $\by_{j}$ are then oppositely-oriented tangent vectors to the hyperbolic geodesic intersecting $P_i$ and $P_{j}$ perpendicularly, so $\by_i\circ\by_{j} < 0$ by Lemma \ref{geometry of inner product}. Furthermore, by this result the hyperbolic cosine of the distance $\ell_{ij}$ from $P_i$ to $P_j$, which we will here denote $L_{ij}$, satisfies $L_{ij} = -\by_i\circ\by_j$ for each $i<j\in\{1,2,3\}$. Also by Lemma \ref{to the hyperplane!}, $\by_i\circ\bx_4 <0$ for each $i\in\{1,2,3\}$, and the signed distance $h_{i4}$ from $P_i$ to $B_4$ satisfies $e^{h_{i4}} = -\by_i\circ\bx_4$.

Since the $P_i$ are labeled arbitrarily, we may take $(i,j,k) = (1,2,3)$ without loss of generality. We will also prove the Lemma's conclusion for particular parametrizations of $\tilde\lambda_{12}$ and $\tilde\lambda_{34}$ below. This will imply the general result, since for any other parametrization the resulting $D(s,t)$ will be obtained from this one by precomposing with a translation of $\mathbb{R}^2$ and a map of the form $(s,t) \mapsto (\pm s,\pm t)$. 
By Lemma \ref{perp point}, since $P_1\subset H_2$ the point of intersection $\tilde\lambda_{12}\cap P_1$ is: 
\[ \bv_1 = -\frac{(\by_1\circ\by_2)\by_1 - \by_2}{\sqrt{(\by_1\circ\by_2)^2-1}} = \frac{L_{12}\,\by_1+\by_2}{\sinh \ell_{12}}. \] 
Using the parametrization (\ref{arbitrary}) with $\bv_1$ as starting point and tangent vector $-\by_1$ (since $\by_1$ is outward-pointing from $H_1$), we obtain $\tilde\lambda_{12}(s) = \cosh s\,\bv_1-\sinh s\,\by_1$, an arclength parametrization for $\tilde\lambda_{12}$ having $\tilde\lambda_{12}(0) = \bv_1$ and with $\tilde\lambda_{12}(s)\in H_1$ for small $s>0$. By Lemma \ref{to the hyperplane!}, $\tilde\lambda_{34}$ is parametrized by $\tilde\lambda_{34}(t) = e^{-h_{34}}\cosh t\,\bx_4 + e^{-t}\by_3$.

$D(s,t)$ described above thus satisfies
\begin{align}\label{archibald} D(s,t) & = - (\cosh s\,\bv_1 - \sinh s\, \by_1)\circ(e^{-h_{34}}\cosh t\,\bx_4 + e^{-t}\,\by_3) \nonumber\\
	& = -e^{-h_{34}}\cosh s\cosh t\, \bv_1\circ\bx_4 - \cosh se^{-t}\, (\bv_1\circ\by_3) \\
	& \qquad\qquad\qquad + e^{-h_{34}}\sinh s \cosh t\,\by_1\circ\bx_4 + \sinh se^{-t}\, (\by_1\circ\by_3) \nonumber \\
	& = C_{++} e^se^t + C_{+-}e^se^{-t} + C_{-+} e^{-s}e^t + C_{--} e^{-s}e^{-t}, \nonumber 
\end{align}
where:\begin{align}\label{dribble}
	& C_{++} = \frac{e^{-h_{34}}}{4}\left(-\bv_1\circ\bx_4+\by_1\circ\bx_4\right) && C_{+-} = C_{++} + \frac{1}{2}\left(-\bv_1\circ\by_3 + \by_1\circ\by_3\right) \\
	& C_{-+} = \frac{e^{-h_{34}}}{4}\left(-\bv_1\circ\bx_4 - \by_1\circ\bx_4\right) && C_{--} = C_{-+} + \frac{1}{2}\left(-\bv_1\circ\by_3 - \by_1\circ\by_3\right) \nonumber
\end{align}
We claim that $0 < C_{++} < C_{+-}, C_{-+} < C_{--}$. Substituting into the formulas above, and recalling that $L_{12} > 1$, being the hyperbolic cosine of a positive length, we obtain:\begin{align*}
	C_{++}  = \frac{e^{-h_{34}}}{4} & \left( e^{h_{14}}\left(\frac{L_{12}}{\sqrt{L_{12}^2-1}} - 1\right) + \frac{e^{h_{24}}}{\sqrt{L_{12}^2-1}} \right) > 0 \\
	C_{+-} - C_{++} = \frac{1}{2} & \left( L_{13}\left(\frac{L_{12}}{\sqrt{L_{12}^2-1}} - 1 \right) + \frac{L_{23}}{\sqrt{L_{12}^2-1}} \right) > 0 \end{align*}
The formulas for $C_{-+}$ and $C_{++}$ are similar to those above, but with negative signs flipped to positive, from which it follows that $C_{-+} > C_{++}$ and that $C_{--}$ is the largest of the four. 

The claim is thus proved, so Lemma \ref{or test sharc} directly implies the present result's conclusion (1). Since the $C_{\pm\pm}$ all depend only on the distances $\ell_{i'j'}$ and $h_{k'4}$, that result's description of the critical point $(s_0,t_0)$ of $D(s,t)$ in terms of these coefficients, plus the description above of $D(s,t)$ itself, imply that the minimum value $D(s_0,t_0)$ depends only on these lengths---the present result's conclusion (2). Also by Lemma \ref{or test sharc}, $s_0$ and $t_0$ are both positive, so given the parameterizations specified for $\tilde\lambda_{12}(s)$ and $\tilde\lambda_{34}(t)$, this implies that $\tilde\lambda_{12}(s_0)\in H_1$ and $\tilde\lambda_{34}(t_0)\in H_3$.

Swapping the indices $1$ and $2$ in the argument above, while fixing $3$ and $4$, has the effect of re-parametrizing $\tilde\lambda_{12}$ and shows that $\tilde\lambda_{12}(s_0)$---representing the same closest point to $\tilde\lambda_{34}$---also lies in $H_2$. The present result's conclusion (3) follows, as does the fact (part of conclusion (4)) that the function $T(x,y;a,b;c,d)$ defined there is invariant under exchanging $a$ with $b$ and $c$ with $d$. (This is the effect on the inputs of the index swap $1\leftrightarrow 2$ for the $P_i$.)

For the final part of conclusion (4), that the function $T$ is increasing in each of $L_{13}$, $L_{23}$, $e^{h_{14}}$, and $e^{h_{24}}$, we rearrange the terms of equation (\ref{archibald}) and divide by $\cosh s$, yielding:
\begin{align}\label{D one} \frac{D(s,t)}{\cosh s} = \left(\cosh t\,\frac{e^{h_{14}}}{e^{h_{34}}} + e^{-t}\,L_{13}\right)\left(\coth\ell_{12}-\tanh s\right) + \frac{1}{\sinh\ell_{12}} \left( \cosh t \frac{e^{h_{24}}}{e^{h_{34}}} + e^{-t}L_{23} \right). \end{align}
From this we see that ``$D$'', now taken to represent a \emph{family} of functions of $(s,t)$ parametrized by $(L_{12},e^{h_{34}},L_{13},L_{23},e^{h_{14}},e^{h_{24}})$, increases pointwise with each of the last four quantities for any fixed $(s,t)$; therefore its absolute minimum $T(x,y;L_{13},L_{23};e^{h_{14}},e^{h_{24}})$ does as well.
\end{proof}

\begin{proposition}\label{tetra transversals one} Suppose $P_1$, $P_2$, and $P_3$ are pairwise disjoint and non-parallel planes in $\mathbb{H}^3$, and $B_4$ is a horoball defined by a vector $\bx_4$ not in the ideal boundary of any $P_i$, such that for each $i<4$, a single half-space $H_i$ bounded by $P_i$ contains $P_j$ for all $j\ne i<4$ and has $\bx_4$ in its ideal boundary. For any $i< j\in\{1,2,3\}$, let $T_1(x,y;a,b;c,d)$ record the transversal length of $\Delta$ as in Lemma \ref{shortest arc one}, where $x = \cosh \ell_{ij}$, $y = e^{h_{k4}}$, $a$ and $b$ are hyperbolic cosines of the other internal edge lengths, and $c$ and $d$ are of the form $e^{h_{i4}}$ for the signed distance between $B_4$ and the external faces in $P_i$, $i\ne k$. If $a$ and $b$ are both at least $L$ and $c$ and $d$ are both at least $H$, for some fixed $L>1$ and $H\ge 1$ then:
\[ \cosh T_1(x,y;a,b;c,d) \ge \sqrt{\frac{2H}{y(x-1)} \left(\frac{H}{y}+2L\right)}, \]
with equality if $a = b = L$ and $c = d = H$.
\end{proposition}

\begin{proof} We compute the equality case first. Taking $(i,j,k) = (1,2,3)$ as in the proof of Lemma \ref{shortest arc one}, so that $x = L_{12}$, $y = e^{h_{34}}$, and taking $a = b = L$ and $c=d=H$, the values for $C_{++}$ and $C_{+-}$ described in (\ref{dribble}) become:\begin{align*}
	& C_{++} = \frac{H}{4y}\left(\sqrt{\frac{x+1}{x-1}} - 1 \right) && C_{+-} = \left(\frac{H}{4y}+\frac{L}{2}\right)\left(\sqrt{\frac{x+1}{x-1}} - 1 \right)
\end{align*}
Values of $C_{-+}$ and $C_{--}$ are like those of $C_{++}$ and $C_{+-}$, respectively, but with negative signs flipped to positive. We therefore obtain:
\[ C_{++}C_{--} = C_{+-}C_{-+} = \frac{H}{4y} \left(\frac{H}{4y}+\frac{L}{2}\right) \frac{2}{x-1}. \]
Therefore by (\ref{crit loc}), in this case $D(s_0,t_0) = 4\sqrt{C_{++}C_{--}}$ is given by the formula in the Proposition statement. That it bounds the value of $\cosh T_1$ below, for $a,b\ge L$ and $c,d\ge H$, follows from the fact proved in Lemma \ref{shortest arc one}, that $T_1$ is increasing in $a$ and $b$, and separately in $c$ and $d$.
\end{proof}

\subsection{Two horoballs} We now consider the case $k= 2$, of partially truncated tetrahedra with exactly two ideal vertices. Here there are two qualitatively different transversals. Proposition \ref{tetra transversals two} considers the length of a ``ray-ray'' transversal joining a pair of opposite internal edges that each has one ideal endpoint. Proposition \ref{shortest arc two cb} then treats the length of the ``compact-bi-infinite'' transversal joining the unique compact and bi-infinite edges. Notably, we obtain fully general explicit formulas in both cases.

\begin{proposition}\label{tetra transversals two} Suppose $P_1$ and $P_2$ are pairwise disjoint and non-parallel planes in $\mathbb{H}^3$, and $B_3$ and $B_4$ are horoballs respectively defined by vectors $\bx_3$ and $\bx_4$, neither in the ideal boundary of either $P_i$, such that for each $i\in\{1,2\}$, a single half-space $H_i$ bounded by $P_i$ contains $P_{3-i}$ and, in its ideal boundary, each of $\bx_3$ and $\bx_4$. Let $\tilde\lambda_{12}$ be the geodesic intersecting $P_1$ and $P_2$ at right angles; for $i\in\{1,2\}$ and $j\in\{3,4\}$ let $\tilde\lambda_{ij}$ be the perpendicular geodesic to $P_i$ in the direction of $\bx_j$; and let $\tilde\lambda_{34}$ be the geodesic with ideal endpoints at $\bx_3$ and $\bx_4$. Fixing parametrizations $\tilde\lambda_{13}(s)$ and $\tilde\lambda_{24}(t)$ by arclength:\begin{enumerate} 
	\item the function $D(s,t)$ that records the hyperbolic cosine of the distance from $\tilde\lambda_{13}(s)$ to $\tilde\lambda_{24}(t)$ has a unique critical point $(s_0,t_0)$ in $\mathbb{R}^2$, at which it attains an absolute minimum;
	\item the absolute minimum value $D(s_0,t_0)$ depends only on the pairwise distance $\ell_{12}$ between $P_1$ and $P_2$, signed distances $h_{ij}$ from the $P_i$ ($i\in\{1,2\}$) to the $B_j$ ($j\in\{3,4\}$), and on the signed distance $d_{34}$ from $B_3$ to $B_4$; 
	\item the points $\tilde\lambda_{13}(s_0)$ and $\tilde\lambda_{24}(t_0)$ lie in the partially truncated tetrahedron $\Delta$ determined by the $P_i$ and $B_j$ as in Definition \ref{one side prime}, so $D(s_0,t_0)$ is the hyperbolic cosine of a transversal length of $\Delta$; and
	\item written as $T_2^{rr}(x,y;a;b,c;d)$, where $x = e^{h_{13}}$, $y = e^{h_{24}}$, $a = \cosh\ell_{12}$, $b$ and $c$ are of the form $e^{h_{ij}}$ for $(i,j) = (1,4)$ and $(2,3)$, and $d = e^{d_{34}}$, this transversal length is 
increasing in each of $a$, $b$, $c$, and $d$ individually.
\end{enumerate}
For $x$, $y$, $a$, $b$, $c$, and $a$ as above, $T_2^{rr}(x,y;a;b,c;d)$ is explicitly given by:\begin{align*}
	\cosh T_2^{rr}(x,y;a;b,c;d) = \sqrt{\frac{1}{xy}\left(\frac{d}{y} + c\right)\left(\frac{d}{x}+b\right)} + 
		\sqrt{\frac{d}{xy}\left(\frac{d}{xy}+\frac{c}{x} + \frac{b}{y} +  2a\right)}
	\end{align*}
\end{proposition}

\begin{proof} For $i = 1,2$, let $\by_i$ be the outward normal to the half-space $H_i$ bounded by $P_i$ that contains $P_{3-i}$. Then $\by_1$ and $\by_2$ are oppositely-oriented tangent vectors to the geodesic intersecting $P_1$ and $P_2$ perpendicularly, so by Lemma \ref{geometry of inner product} $\by_1\circ\by_2 <0$ and $-\by_1\circ\by_2 = L_{12}\doteq\cosh\ell_{12}$, where $\ell_{12}$ is the shortest distance between $P_1$ and $P_2$. By Lemma \ref{to the hyperplane!}, the geodesic $\tilde\lambda_{13}$ perpendicular to $P_1$ and in the direction of $\bx_3$ has the parametrization $\tilde\lambda_{13}(s) = e^{-h_{13}} \cosh s \bx_3 + e^{-s}\by_1$, where the signed distance $h_{13}$ from $P_1$ to $B_3$ satisfies $e^{h_{13}} = -\by_1\circ\bx_3$ (noting that $\by_1\circ\bx_3<0$, since $\bx_3$ is in the ideal boundary of $H_1$ by hypothesis). The hypothesis and Lemma \ref{to the hyperplane!} likewise imply that for any $i\in\{1,2\}$ and $j\in\{3,4\}$, the signed distance $h_{ij}$ from $P_i$ to $B_j$ satisfies $e^{h_{ij}} = -\by_i\circ\bx_j$, and that the geodesic $\tilde\lambda_{24}$ perpendicular to $P_2$ and in the direction of $\bx_4$ has the parametrization $\tilde\lambda_{24}(t) = e^{-h_{24}}\cosh t \bx_4 + e^{-t}\by_2$. Lemma \ref{to the other horoball!} implies that the signed distance $d_{34}$ from $B_3$ to $B_4$ satisfies $e^{d_{34}} = - \frac{1}{2}\bx_3\circ\bx_4$, so $-\bx_3\circ\bx_4 = 2e^{d_{34}}$.

The hyperbolic cosine $D(s,t)$ of the distance between $\tilde\lambda_{13}(s)$ and $\tilde\lambda_{24}(t)$ therefore satisfies:\begin{align*}
	D(s,t) & = - \left( e^{-h_{13}} \cosh s \bx_3 + e^{-s}\by_1 \right) \circ \left( e^{-h_{24}}\cosh t \bx_4 + e^{-t}\by_2 \right) \\
	& = \cosh s \cosh t \frac{2e^{d_{34}}}{e^{h_{13}+h_{24}}} + \cosh se^{-t} \frac{e^{h_{23}}}{e^{h_{13}}} 
		+  e^{-s}\cosh t \frac{e^{h_{14}}}{e^{h_{24}}} + e^{-s}e^{-t} L_{12} \\
	& = C_{++} e^s e^t + C_{+-} e^s e^{-t} + C_{-+} e^se^{-t} + C_{--} e^{-s}e^{-t},\end{align*}
where:\begin{align*}
	& C_{++} = \frac{1}{2}\frac{e^{d_{34}}}{e^{h_{13}+h_{24}}}, && C_{+-} = C_{++} + \frac{1}{2}\frac{e^{h_{23}}}{e^{h_{13}}}, &
	& C_{-+} = C_{++} + \frac{1}{2}\frac{e^{h_{14}}}{e^{h_{24}}}, 
\end{align*}
and $C_{--} = C_{+-} + C_{-+} - C_{++} + L_{12}$. It is clear that $0 < C_{++} < C_{+-}, C_{-+} < C_{--}$, so by Lemma \ref{or test sharc}, $D(s,t)$ has a unique critical point $(s_0,t_0)$ at which it attains an absolute minimum; and moreover, that each coordinate of this critical point is positive. This establishes the Proposition's assertion (1), and (3) follows from the fact that each of $\tilde\lambda_{12}(s)$ and $\tilde\lambda_{34}(t)$ are parametrized pointing into $\Delta$ at $0$.

Assertion (2), regarding the dependencies of the minimum value $D(s_0,t_0)$, follows from its explicit description in terms of the $C_{\pm\pm}$ in formula (\ref{crit loc}) and their descriptions above---for instance, $C_{++} = D/(4xy)$. The explicit formula $T_2^{rr}(x,y;a;b,c;d)$ for this minimum value follows by direct substitution into (\ref{crit loc}), and assertion (4) regarding its monotonicity in certain variables can be seen by inspection of this formula.\end{proof}

\begin{proposition}\label{shortest arc two cb} Suppose $P_1$ and $P_2$ are pairwise disjoint and non-parallel planes in $\mathbb{H}^3$, and $B_3$ and $B_4$ are horoballs respectively defined by vectors $\bx_3$ and $\bx_4$, neither in the ideal boundary of either $P_i$, such that for each $i\in\{1,2\}$, a single half-space $H_i$ bounded by $P_i$ contains $P_{3-i}$ and, in its ideal boundary, each of $\bx_3$ and $\bx_4$. Let $\tilde\lambda_{12}$ be the geodesic intersecting $P_1$ and $P_2$ at right angles; for $i\in\{1,2\}$ and $j\in\{3,4\}$ let $\tilde\lambda_{ij}$ be the perpendicular geodesic to $P_i$ in the direction of $\bx_j$; and let $\tilde\lambda_{34}$ be the geodesic with ideal endpoints at $\bx_3$ and $\bx_4$. Fixing parametrizations $\tilde\lambda_{12}(s)$ and $\tilde\lambda_{34}(t)$ by arclength:\begin{enumerate} 
	\item the function $D(s,t)$ that records the hyperbolic cosine of the distance from $\tilde\lambda_{12}(s)$ to $\tilde\lambda_{34}(t)$ has a unique critical point $(s_0,t_0)$ in $\mathbb{R}^2$, at which it attains an absolute minimum;
	\item the absolute minimum value $D(s_0,t_0)$ depends only on the pairwise distance $\ell_{12}$ between $P_1$ and $P_2$, signed distances $h_{ij}$ from the $P_i$ ($i\in\{1,2\}$) to the $B_j$ ($j\in\{3,4\}$), and on the signed distance $d_{34}$ from $B_3$ to $B_4$; 
	\item the points $\tilde\lambda_{12}(s_0)$ and $\tilde\lambda_{34}(t_0)$ lie in the partially truncated tetrahedron $\Delta$ determined by the $P_i$ and $B_j$ as in Definition \ref{one side prime}, so $D(s_0,t_0)$ is the hyperbolic cosine of a transversal length of $\Delta$; and
	\item written as $T_2^{cb}(x,y;a,b;c,d)$, where $x = \cosh \ell_{12}$, $y = e^{d_{34}}$, $a = e^{h_{13}}$, $b = e^{h_{14}}$, $c = e^{h_{23}}$, and $d = e^{h_{24}}$, this transversal length is increasing in each of $a$, $b$, $c$, $d$ individually and invariant under any product of two disjoint transpositions of $\{a,b,c,d\}$.
\end{enumerate}
Given explicitly, $\cosh T_2^{cb} (x,y;a,b;c,d)$ equals
\begin{align*}
	\frac{1}{2\sqrt{y}} \left[ \sqrt{\left(\frac{xa+c}{\sqrt{x^2-1}} - a\right)\left(\frac{xb + d}{\sqrt{x^2-1}} + b \right) } + \sqrt{\left(\frac{xa+c}{\sqrt{x^2-1}} + a\right)\left(\frac{xb + d}{\sqrt{x^2-1}} - b \right) }\right]
\end{align*}
\end{proposition}

\begin{proof} For $i = 1$ and $2$, let $\by_i$ be an outward unit normal, in the sense described below Lemma \ref{space-like corresp}, to the half-space $H_i$ bounded by $P_i$ that contains $P_{3-i}$. Thus $\by_1$ and $\by_2$ are oppositely-oriented tangent vectors to the hyperbolic geodesic intersecting $P_1$ and $P_2$ perpendicularly, so $\by_1\circ\by_2 < 0$ by Lemma \ref{geometry of inner product}. Furthermore, by this result the hyperbolic cosine of the distance $\ell_{12}$ from $P_1$ to $P_2$, which we will here denote $L_{12}$, satisfies $L_{12} = -\by_1\circ\by_2$. By Lemma \ref{perp point}, since $P_1\subset H_2$ the point of intersection $\tilde\lambda_{12}\cap P_1$ is: 
\[ \bv_1 = -\frac{(\by_1\circ\by_2)\by_1 - \by_2}{\sqrt{(\by_1\circ\by_2)^2-1}} = \frac{L_{12}\,\by_1+\by_2}{\sinh \ell_{12}}. \] 
Using the parametrization (\ref{arbitrary}) with $\bv_1$ as starting point and tangent vector $-\by_1$ (since $\by_1$ is outward-pointing from $H_1$), we obtain $\tilde\lambda_{12}(s) = \cosh s\,\bv_1-\sinh s\,\by_1$, an arclength parametrization for $\tilde\lambda_{12}$ having $\tilde\lambda_{12}(0) = \bv_1$ and with $\tilde\lambda_{12}(s)\in H_1$ for small $s>0$.

By Lemma \ref{to the hyperplane!}, $\by_i\circ\bx_j <0$ for each $i\in\{1,2\}$ and $j\in\{3,4\}$, and the signed distance $h_{ij}$ from $P_i$ to $B_j$ satisfies $e^{h_{ij}} = -\by_i\circ\bx_j$. By Lemma \ref{to the other horoball!}, the signed distance $d_{34}$ from $B_3$ to $B_4$ satisfies $2e^{d_{34}} = - \bx_3\circ\bx_4$, and the geodesic $\tilde\lambda_{34}$ perpendicular to each horoball and pointing toward $\bx_3$ has parametrization $\tilde\lambda_{34}(t) = \frac{1}{2}e^{-d_{34}/2}\left(e^t\bx_3 + e^{-t}\bx_4\right)$.

The hyperbolic cosine of the distance $D(s,t)$ between $\tilde\lambda_{12}(s)$ and $\tilde\lambda_{34}(t)$ satisfies:\begin{align*}
	D(s,t) & =  - \left( \cosh s\,\bv_1-\sinh s\,\by_1 \right) \circ \frac{e^{-d_{34}/2}}{2}\left(e^t\bx_3 + e^{-t}\bx_4\right) \\
	& = \frac{e^{-d_{34}/2}}{4}\left[ e^s e^t \left(\frac{L_{12}\,e^{h_{13}}+e^{h_{23}}}{\sinh \ell_{12}} - e^{h_{13}} \right) + e^s e^{-t} \left( \frac{L_{12}\,e^{h_{14}}+e^{h_{24}}}{\sinh \ell_{12}} - e^{h_{14}} \right) \right. \\
	& \hspace{1in} \left. + e^{-s} e^t \left(\frac{L_{12}\,e^{h_{13}}+e^{h_{23}}}{\sinh \ell_{12}} + e^{h_{13}} \right)
	 + e^{-s} e^{-t} \left( \frac{L_{12}\,e^{h_{14}}+e^{h_{24}}}{\sinh \ell_{12}} + e^{h_{14}} \right) \right]
	\end{align*}
Let $C_{++}$ be the coefficient of $e^se^t$ in the expression for $D(s,t)$ above, and name the coefficients of the other addends as $C_{+-}$, $C_{-+}$, $C_{--}$ correspondingly. We have:
\[ C_{++} =  \frac{e^{-d_{34}/2}}{4}\left[ e^{h_{13}}\left(\frac{L_{12}}{\sinh \ell_{12}} - 1 \right) +  \frac{e^{h_{23}}}{\sinh \ell_{12}} \right] > 0. \]
We correspondingly observe that $C_{+-}>0$, that $C_{-+}>C_{++}$, and that $C_{--}>C_{+-}$. Therefore by Lemma \ref{or test sharc}, $D(s,t)$ is convex and attains an absolute minimum at a unique critical point $(s_0,t_0)$. Plugging the computed coefficients $C_{\pm\pm}$ into the formula (\ref{crit loc}) for $D(s_0,t_0)$ yields:\begin{align*}
	D(s_0,t_0) & = \frac{1}{2e^{d_{34}/2}}\left[ \sqrt{\left(\frac{L_{12}\,e^{h_{13}}+e^{h_{23}}}{\sinh \ell_{12}} - e^{h_{13}} \right) \left( \frac{L_{12}\,e^{h_{14}}+e^{h_{24}}}{\sinh \ell_{12}} + e^{h_{14}} \right) } \right. \\
	& \hspace{1in} \left. + \sqrt{  \left(\frac{L_{12}\,e^{h_{13}}+e^{h_{23}}}{\sinh \ell_{12}} + e^{h_{13}} \right) \left( \frac{L_{12}\,e^{h_{14}}+e^{h_{24}}}{\sinh \ell_{12}} - e^{h_{14}} \right) } \right]
	\end{align*}
This yields the explicit formula for $T_2^{cb}(x,y;a,b,c,d)$ given in the Proposition statement. The symmetry property of $T_2^{cb}$ described in (4) there reflects that swapping the labels of $P_1$ and $P_2$, and/or those of $B_3$ and $B_4$, does not change the minimum value. The effect on the variables $a$, $b$, $c$, $d$  of doing either or both of these label swaps is to act as a product of disjoint transpositions. That $T_2(x,y;a,b,c,d)$ is increasing in each of its final four variables follows from the fact that $D$, considered as a family of functions parametrized by $(x,y;a,b,c,d)$, is pointwise increasing.

For the Proposition's assertion (3), we note that since $C_{-+}>C_{++}$, and that $C_{--}>C_{+-}$, from the formula (\ref{crit loc}) we obtain $e^{s_0} > 1 \Rightarrow s_0 > 0$. It follows that $s_0>0$, and hence that $\tilde\lambda_{12}(s_0)$ lies in the same half-space bounded by $P_1$ as the tetrahedron $\Delta$. Since the assignment of indices to $P_1$ and $P_2$ was arbitrary, it follows that $\lambda_{12}(s_0)$ also lies in the same half-space bounded by $P_2$ as the tetrahedron $\Delta$. Therefore $T_2^{cb}(x,y;a,b,c,d)$ is a transversal length of $\Delta$.
\end{proof}

\subsection{Three horoballs} The final partially truncated case.

\begin{proposition}\label{tetra transversals three} Suppose $P_1$ is a plane in $\mathbb{H}^3$ bounding a half-space $H_1$, and that $B_2$, $B_3$ and $B_4$ are horoballs respectively defined by pairwise linearly independent vectors $\bx_2$, $\bx_3$ and $\bx_4$, such that each $\bx_i$ lies in the ideal boundary of $H_1$ but not of $P_1$, for $i>1$. Let $\tilde\lambda_{12}$ be the geodesic with ideal endpoint at $\bx_2$ that intersects $P_1$ perpendicularly, and let $\tilde\lambda_{34}$ be the geodesic with its ideal endpoints at $\bx_3$ and $\bx_4$. Fixing parametrizations $\tilde\lambda_{12}(s)$ and $\tilde\lambda_{34}(t)$ by arclength:\begin{enumerate} 
	\item the function $D(s,t)$ that records the hyperbolic cosine of the distance from $\tilde\lambda_{12}(s)$ to $\tilde\lambda_{34}(t)$ has a unique critical point $(s_0,t_0)$ in $\mathbb{R}^2$, at which it attains an absolute minimum;
	\item the absolute minimum value $D(s_0,t_0)$ depends only on the pairwise signed distances $h_{1j}$ from $P_1$ to the $B_j$ ($j\in\{2,3,4\}$), and $d_{jk}$ from $B_j$ to $B_k$, $j < k\in\{2,3,4\}$; 
	\item the points $\tilde\lambda_{12}(s_0)$ and $\tilde\lambda_{34}(t_0)$ lie in the partially truncated tetrahedron $\Delta$ determined by $P_1$ and the $B_j$ as in Definition \ref{one side prime}, so $D(s_0,t_0)$ is the hyperbolic cosine of a transversal length of $\Delta$; and
	\item written as $T_3(x,y;a,b;c,d)$, where $x = e^{h_{12}}$, $y = e^{d_{34}}$, $a = e^{h_{13}}$, $b = e^{h_{14}}$, $c = e^{h_{23}}$, and $d = e^{h_{24}}$, this transversal length is increasing in each of $a$, $b$, $c$, $d$ individually and invariant under the involution exchanging $a$ with $b$ and $c$ with $d$.
\end{enumerate}
Given explicitly, $\displaystyle{\cosh T_3(x,y;a,b;c,d) = \frac{1}{\sqrt{xy}} \left[ \sqrt{c\left(\frac{d}{x}+b\right)} + \sqrt{d\left(\frac{c}{x}+a\right)}. \right] }$
\end{proposition}

\begin{proof} Let $\by_1$ be an outward unit normal, in the sense described below Lemma \ref{space-like corresp}, to the half-space $H_1$ bounded by $P_1$ that contains the $\bx_j$ for $j=2,3,4$. By Lemma \ref{to the hyperplane!}, the geodesic $\tilde\lambda_{12}$ perpendicular to $P_1$ and in the direction of $\bx_2$ has the parametrization $\tilde\lambda_{12}(s) = e^{-h_{12}} \cosh s \bx_2 + e^{-s}\by_1$, where the signed distance $h_{12}$ from $P_1$ to $B_2$ satisfies $e^{h_{12}} = -\by_1\circ\bx_2$ (noting that $\by_1\circ\bx_2<0$, since $\bx_2$ is in the ideal boundary of $H_1$ by hypothesis). The hypothesis and Lemma \ref{to the hyperplane!} likewise imply that for any $j\in\{3,4\}$, the signed distance $h_{1j}$ from $P_1$ to $B_j$ satisfies $e^{h_{1j}} = -\by_1\circ\bx_j$,

By Lemma \ref{to the other horoball!}, for $j<k\in\{2,3,4\}$ the signed distance $d_{jk}$ from $B_j$ to $B_k$ satisfies $2e^{d_{jk}} = - \bx_j\circ\bx_k$. Furthermore, the geodesic $\tilde\lambda_{34}$ perpendicular to the horoballs $B_3$ and $B_4$ and pointing toward $\bx_3$ has parametrization $\tilde\lambda_{34}(t) = \frac{1}{2}e^{-d_{34}/2}\left(e^t\bx_3 + e^{-t}\bx_4\right)$. The hyperbolic cosine $D(s,t)$ of the distance between $\tilde\lambda_{12}(s)$ and $\tilde\lambda_{34}(t)$ therefore satisfies:\begin{align*}
	D(s,t) & = - \left( e^{-h_{12}} \cosh s \bx_2 + e^{-s}\by_1 \right) \circ \left(  \frac{1}{2}e^{-d_{34}/2}\left(e^t\bx_3 + e^{-t}\bx_4\right)\right) \\
		& = \frac{1}{2e^{d_{34}/2}}\left[ \cosh s e^t \frac{2e^{d_{23}}}{e^{h_{12}}} + \cosh s e^{-t} \frac{2e^{d_{24}}}{e^{h_{12}}} + e^{-s} e^t e^{h_{13}} + e^{-s}e^{-t} e^{h_{14}} \right] \\
		& = C_{++} e^se^t + C_{+-}e^se^{-t} + C_{-+} e^{-s}e^t + C_{--} e^{-s}e^{-t},
\end{align*}
where:\begin{align*}
	& C_{++} = \frac{e^{d_{23}}}{2e^{h_{12}}e^{d_{34}/2}} && C_{+-} = \frac{e^{d_{24}}}{2e^{h_{12}}e^{d_{34}/2}} && C_{-+} = C_{++} + \frac{e^{h_{13}}}{2e^{d_{34}/2}} && C_{--} = C_{+-} +  \frac{e^{h_{14}}}{2e^{d_{34}/2}} 
\end{align*}
Since each coefficient $C_{\pm\pm}$ is visibly positive, by Lemma \ref{or test sharc} $D(s,t)$ is convex and attains an absolute minimum at a unique critical point $(s_0,t_0)$. This establishes the present result's assertion (1). Assertion (2) follows from the description of the minimum value $D(s_0,t_0)$, in (\ref{crit loc}), and those of the coefficients $C_{\pm\pm}$ above. We note that $C_{--} > C_{+-}$ and $C_{-+} > C_{++}$, so by the descriptions in (\ref{crit loc}) we have $e^{s_0} > 1 \Rightarrow s_0 > 0$. Since $\tilde\lambda_{12}$ is parametrized pointing into $H_1$ and so that $\tilde\lambda_{12}(0) \in P_1$, this implies that $\tilde\lambda_{12}(s_0)$ lies in $H_1$ and hence in $\Delta$. Therefore $D(s_0,t_0)$ is the hyperbolic cosine of a transversal length of $\Delta$, confirming assertion (3).

The explicit formula for $T_3(x,y;a,b;c,d) = D(s_0,t_0)$ given in the present result's statement is obtained by directly applying the formula from (\ref{crit loc}) to the coefficients $C_{\pm\pm}$ computed above. Its properties listed in assertion (4) are visible from this formula.
\end{proof}

\subsection{Four horoballs} Our techniques also apply to the case of (fully) ideal tetrahedra.

\begin{proposition}\label{ideal tet trans} Let $\Delta\subset\mathbb{H}^3$ be the ideal tetrahedron determined by positive, pairwise-linearly independent light-like vectors $\bx_1$, $\bx_2$, $\bx_3$, $\bx_4$, and for each $i<j$ let $d_{ij}$ be the signed distance between the horoballs $B_i$ and $B_j$ determined by $\bx_i$ and $\bx_j$. Let $\tilde\lambda_{12}$ be the geodesic joining $\bx_1$ to $\bx_2$, and let $\tilde\lambda_{34}$ be the geodesic joining $\bx_3$ to $\bx_4$. Taking $x = e^{d_{12}}$, $y = e^{d_{34}}$, $a = e^{d_{13}}$, $b = e^{d_{14}}$, $c = e^{d_{23}}$, and $d = e^{d_{24}}$, the length of the transversal of $\Delta$ joining $\tilde\lambda_{12}$ to $\tilde\lambda_{34}$ is given by
\[ \cosh T_4(x,y;a,b,c,d) = \frac{\sqrt{ad} + \sqrt{bc}}{\sqrt{xy}}. \]
\end{proposition}

Here are a couple of sanity checks for this formula:\begin{itemize}
	\item In the case that all $\bx_i$ lie in a single plane, ie.~if the ideal tetrahedron they determine degenerates to a quadrilateral, in which the line through $\bx_1$ and $\bx_2$ separates $\bx_3$ from $\bx_4$, applying Penner's ``ideal Ptolemy theorem'' \cite[Prop.~2.6(a)]{Penner} to the formula gives $\cosh T_4(x,y;a,b,c,d) = 1$, so the value of $T_4$ is $0$. This is correct, reflecting that in this case the edges $\lambda_{12}$ and $\lambda_{34}$ are crossing diagonals.
	\item The transversal length of a regular ideal tetrahedron with a maximal, fully symmetric horoball packing at its vertices is determined by the formula as $\cosh T_4(1,1;1,1,1,1) = 2$. This can be confirmed by the angle of parallelism (see eg.~\cite[Th.~7.9.1(ii)]{Beardon}), since by symmetry considerations the transversal is the compact edge of a hyperbolic triangle with angles $0$, $\pi/2$, and $\pi/6$.\end{itemize}

\begin{remarks}\label{rescale horoball} Note that $T_4$ is invariant under any product of disjoint transpositions of its inputs $\{a,b,c,d\}$. As in prior results here, this reflects its insensitivity to swapping the label of $\bx_1$ with $\bx_2$ and/or of $\bx_3$ with $\bx_4$. 

Note that $T_4$ is \emph{also} invariant under transforming $(a,b,x)\mapsto t(a,b,x)$ for any $t\in(0,\infty)$. This records the effect on the inputs to $T_4$ of rescaling $\bx_1$; ie.~of making a different choice of horoball centered at the same ideal point $[\bx_1]$. Correspondingly, $T_4$ is invariant under transformations $(c,d,x)\mapsto u(c,d,x)$, $(a,c,y)\mapsto v(a,c,y)$, and $(b,d,y)\mapsto w(b,d,y)$ for any $u$, $v$, or $w\in(0,\infty)$.

The upshot of the previous paragraph is that while Proposition \ref{ideal tet trans} requires the input of a set of horoball neighborhoods in order to compute transversal length, its output is independent of the particular neighborhoods chosen. The prior results of this section also exhibit analogous invariance under horoball rescaling.\end{remarks}

\begin{proof} We parametrize $\tilde\lambda_{12}$ pointing toward $\bx_1$ as $\tilde\lambda_{12}(s) = \frac{1}{2}e^{-d_{12}/2}\left(e^s\bx_1 + e^{-s}\bx_2\right)$, and $\tilde\lambda_{34}$ pointing toward $\bx_3$ as $\tilde\lambda_{34}(t) = \frac{1}{2}e^{-d_{34}/2}\left(e^t\bx_3 + e^{-t}\bx_4\right)$, recalling Lemma \ref{to the other horoball!}. The hyperbolic cosine of the distance from $\tilde\lambda_{12}(s)$ to $\tilde\lambda_{34}(t)$ is given by:\begin{align*}
	D(s,t) & = -\frac{1}{2}e^{-d_{12}/2}\left(e^s\bx_1 + e^{-s}\bx_2\right) \circ \frac{1}{2}e^{-d_{34}/2}\left(e^t\bx_3 + e^{-t}\bx_4\right) \\
		& = \frac{1}{4e^{d_{12}/2}e^{d_{34}/2}}\left[ 2e^{d_{13}}e^s e^t + 2e^{d_{14}}e^se^{-t} + 2e^{d_{23}}e^{-s}e^t + 2e^{d_{24}}e^{-s}e^{-t}\right] 
\end{align*}
As the coefficients of the summands $e^se^t$, $e^se^{-t}$, $e^{-s}e^t$, and $e^{-s}e^{-t}$ are all positive, Lemma \ref{or test sharc} implies that $D(s,t)$ is strictly convex and attains an absolute minimum at a unique critical point $(s_0,t_0)$. Furthermore, the formula (\ref{crit loc}) gives the minimum value $D(s_0,t_0)$ as the quantity recorded as $T_4(x,y;a,b,c,d)$ above.\end{proof}

\bibliographystyle{plain}

\end{document}